\documentclass[11pt, reqno, english]{amsart}  
\usepackage[colorlinks=true,urlcolor=blue,linkcolor=red,citecolor=magenta]{hyperref}
\usepackage{amsmath,amssymb,color,url,epstopdf,mathtools,amsthm, paralist}
\usepackage[mathscr]{euscript}
\usepackage{float}
\usepackage{graphicx}
\usepackage{tikz}
\usepackage{tikz-cd}
\usepackage{epsfig}
\usepackage[justification=centering]{caption}

\usepackage[left=1.2in,right=1.2in,top=1.2in,bottom=1.2in]{geometry}

\newtheorem{thm}{Theorem}[section]
\newtheorem{lem}[thm]{Lemma}
\newtheorem{prop}[thm]{Proposition}
\newtheorem{cor}[thm]{Corollary}
\newtheorem{question}[thm]{Question}

\newtheoremstyle{named}{}{}{\itshape}{}{\bfseries}{.}{.5em}{#1 \thmnote{#3}}
\theoremstyle{named}
\newtheorem*{namedthm}{Theorem}
\newtheorem*{namedcor}{Corollary}
\theoremstyle{remark}
\newtheorem{definition}[thm]{Definition}
\newtheorem{example}[thm]{Example}
\newtheorem*{remark}{Remark}
\newtheorem*{claim}{Claim}

\author{Michael Harrison}
\address[MH]{Institute for Advanced Study, 1 Einstein Drive, Princeton, NJ  08540}
\email{mah5044@gmail.com} 

\thanks{Portions of this work were completed at Mathematisches Forschungsinstitut Oberwolfach, where the author was supported by an Oberwolfach Leibniz Fellowship.  Other portions of this work were completed while the author was in residence at the Institute for Advanced Study and supported by the National Science
Foundation under Grant No. DMS-1926686.  The author is grateful to both institutes for their hospitality.}

\title{Skew and sphere fibrations}

\newcommand{\R}{\mathbb R}
\newcommand{\C}{\mathbb C}
\newcommand{\N}{\mathbb N}

\newcommand{\Ham}{\mathbb H}
\newcommand{\Oct}{\mathbb O}

\newcommand{\Hom}{\textnormal{Hom}}
\newcommand{\Span}{\textnormal{Span}}
\newcommand{\Ker}{\textnormal{Ker}}
\newcommand{\Id}{\textnormal{I}}
\newcommand{\rank}{\textnormal{rank}}

\newcommand{\affgrass}{AG_k(n)}
\newcommand{\grass}{G_k(n)}
\newcommand{\grassy}{G_{k+1}(n+1)}
\newcommand{\bee}{\mathscr{B}}
\newcommand{\Cee}{\mathscr{C}}
\begin{document}

\maketitle

\begin{abstract}
A great sphere fibration is a sphere bundle with total space $S^n$ and fibers which are great $k$-spheres.  Given a smooth great sphere fibration, the central projection to any tangent hyperplane yields a \emph{nondegenerate} fibration of $\R^n$ by pairwise skew, affine copies of $\R^k$ (though not all nondegenerate fibrations can arise in this way).  Here we study the topology and geometry of nondegenerate fibrations,  we show that every nondegenerate fibration satisfies a notion of Continuity at Infinity, and we prove several classification results. These results allow us to determine, in certain dimensions, precisely which nondegenerate fibrations correspond to great sphere fibrations via the central projection.  We use this correspondence to reprove a number of recent results about sphere fibrations in the simpler, more explicit setting of nondegenerate fibrations.  For example, we show that every germ of a nondegenerate fibration extends to a global fibration, and we study the relationship between nondegenerate line fibrations and contact structures in odd-dimensional Euclidean space.   We conclude with a number of partial results, in hopes that the continued study of nondegenerate fibrations, together with their correspondence with sphere fibrations, will yield new insights towards the unsolved classification problems for sphere fibrations.
\end{abstract}

\section{Introduction}

\subsection{Motivation}
A \emph{(great) sphere fibration} is a sphere bundle with total space $S^n$ and fibers which are oriented great $k$-spheres.  Algebraic topology imposes strong restrictions on the possible dimensions $k$ and $n$ for which such fibrations may exist.  In particular, the only possible fiber dimensions are $0, 1, 3$, and $7$, which leads to the following complete list of dimensions for great sphere fibrations:
\begin{compactitem}
\item the sphere $S^0$ fibers every sphere,
\item the sphere $S^1$ fibers odd-dimensional spheres,
\item the sphere $S^3$ fibers the spheres $S^{4k-1}$, and
\item the sphere $S^7$ fibers $S^{15}$.
\end{compactitem}

Examples of sphere fibrations are the \emph{Hopf fibrations}.  The Hopf fibrations of odd-dimensional spheres $S^{2k-1}$ by great circles arise by choosing an orthogonal complex structure on $\R^{2k}$ and intersecting the unit sphere with all of the complex lines in $\R^{2k}$.  Similar constructions using quaternionic and octonionic structures yield the Hopf fibrations with higher-dimensional great-sphere fibers.  A thorough yet approachable treatment of the geometry of the Hopf fibrations may be found in \cite{GluckWarnerZiller}.  %As formulated by Petro \cite{Petro}, there are three looming questions in the theory of sphere fibrations.

Studies of great sphere fibrations have focused primarily on three questions:
\begin{question} \
\label{ques:main}

\begin{compactenum}[a)]
\item Is every sphere fibration topologically (smoothly) equivalent to a Hopf fibration?  That is, given a continuous (smooth) sphere fibration $S^k \to S^n$, does there exist a fiber-preserving homeomorphism (diffeomorphism) of $S^n$ taking the fibration to a Hopf fibration?
\item Does there exist a path (in the space of sphere fibrations) beginning at any given sphere fibration and ending at a Hopf fibration? 
\item Is there a deformation retract from the space of all sphere fibrations to its subspace of Hopf fibrations?
\end{compactenum}
\end{question}

The first question is of particular interest because it would settle certain special cases of the Blaschke conjecture; see \cite{McKay2} for an excellent summary of the current progress on this conjecture.

In their seminal work \cite{GluckWarner}, Gluck and Warner answered all three questions positively for great circle fibrations of $S^3$.  Gluck, Warner, and Yang \cite{GluckWarnerYang} answered the first question positively for fibrations of $S^7$ by great sphere copies of $S^3$ and for fibrations of $S^{15}$ by great sphere copies of $S^7$.  

Unfortunately, many of the subsequent treatments of these questions contain errors.    The smooth equivalence referenced in the first question was conjectured for great circle fibrations of higher-dimensional spheres by Yang.  Several attempted proofs were published, but eventually flaws were found in each.  The history is well-catalogued in the Introduction of \cite{McKay} (note that the arXiv version has been updated since the published version, to correct some errors pointed out to McKay by Ballmann and Grove).   McKay had claimed a positive answer to Question \ref{ques:main}(a) for great circle fibrations in any dimension, but in the end his argument holds only for great circle fibrations which satisfy a certain nondegeneracy condition (it is unknown if there are any great circle fibrations which do not satisfy this condition).   If nothing else, past studies of these three questions show that the topology and geometry of great sphere fibrations is very subtle, and  extreme care is necessary in their study.

In an attempt to revitalize the study of great sphere fibrations, we have developed machinery which allows us to study these objects more explicitly.  Given a fibered sphere $S^n$, positioned as the unit sphere in $\R^{n+1}$, the central projection of the open upper hemisphere of $S^n$ to the tangent hyperplane $\R^n$ at the north pole of $S^n$ sends great spheres $S^k$ to affine subspaces $\R^k$ of $\R^n$.  Indeed, each great sphere fiber $S^k$ is the intersection of $S^n$ with a linear copy of $\R^{k+1} \subset \R^{n+1}$, and the image of the fiber $S^k$ under the central projection %(more precisely: the image of the portion lying in the upper hemisphere of $S^n$)
is the intersection of that $\R^{k+1}$ with the tangent hyperplane $\R^n$.  The resulting affine subspaces $\R^k$ do not intersect, since their preimage great hemispheres do not intersect.  Moreover, because the spherical fibers do not intersect on the equator of $S^n$, the copies of $\R^k$ do not meet at infinity; that is, no two fibers contain parallel lines.  In this way, every sphere fibration induces a fibration of $\R^n$ by pairwise skew affine $k$-planes,  so the space of all sphere fibrations sits naturally inside the space of these so-called \emph{skew fibrations}.   The skew line fibration depicted in Figure \ref{fig:skewlines} is the image, under central projection, of the Hopf fibration of $S^3$.

\begin{figure}[h!t]
\centerline{
\includegraphics[width=2.5in]{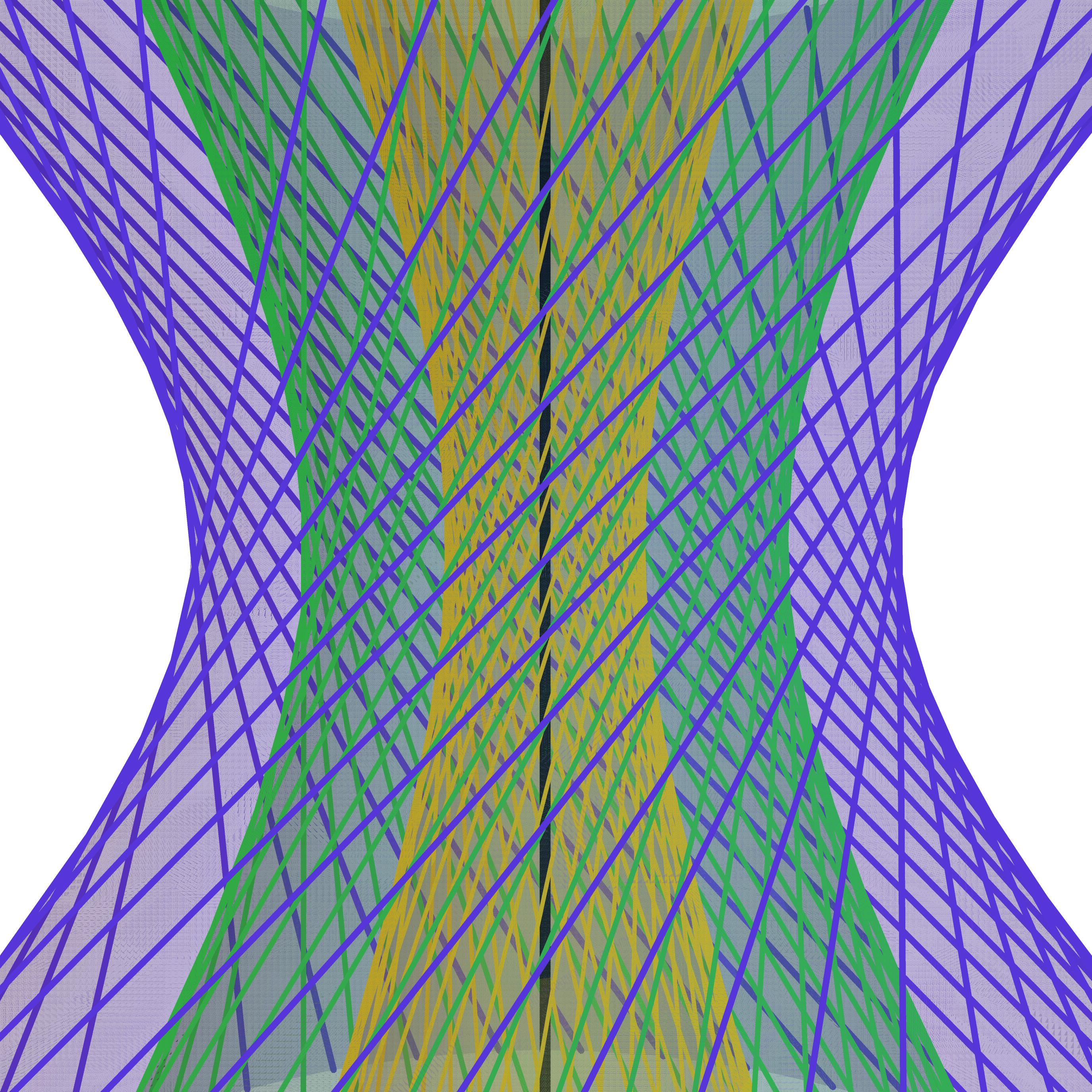}
}
\caption{Skew Hopf fibration of $\R^3$}
\label{fig:skewlines}
\end{figure}

It is worth noting that there exist skew fibrations which do not correspond via central projection to any great sphere fibration, and in fact, there exist many pairs of dimensions $k$ and $n$ which admit skew fibrations but no sphere fibrations; see Section \ref{sec:prelim} for details.

In \cite{Harrison}, we developed the theory of skew line fibrations of $\R^3$ to contribute to the study of great circle fibrations.  In particular, we answered all three parts of Question \ref{ques:main} affirmatively for skew fibrations of $\R^3$, and in the process we provided an alternative method for answering all three questions affirmatively for great circle fibrations of $S^3$.   Skew line fibrations of $\R^3$ have also been studied for reasons not directly related to the study of great circle fibrations, as the theory is interesting in its own right.  Salvai gave a geometric classification of smooth line fibrations of $\R^3$ \cite{Salvai}; whereas the relationship between line fibrations of $\R^3$ and contact structures on $\R^3$ was studied by the author \cite{HarrisonAGT, HarrisonBLMS} and by Becker and Geiges \cite{BeckerGeiges}.  %Geodesic fibrations of hyperbolic space were studied by Godoy and Salvai \cite{GodoySalvai}.

%In fact,  we showed that the deformation retract from the space of all skew fibrations to the Hopf fibrations can be written explicitly as a family of straight-line homotopies.   With this in mind, it is reasonable to expect that a thorough study of skew fibrations would yield explicit and new results in the theory of sphere fibrations.

Despite the attention received by line fibrations of $\R^3$, there has been no attempt to classify the spaces of skew fibrations in higher dimensions. Our present goal is twofold: to initiate the study of higher-dimensional skew fibrations, and to carefully examine the relationship between skew and spherical fibrations in higher dimensions.   We find that skew fibrations are simpler to describe than spherical fibrations, and we see that some standard results for great sphere fibrations are more accessible for skew fibrations.  In certain cases we find that these results and explicit constructions can then be imported to the language of sphere fibrations.  It is reasonable to believe that after the core ideas for skew fibrations have been sufficiently developed, these objects become an important aspect of future discussions involving great sphere fibrations.  This article is a first step towards this broad goal.

\subsection{Statement of Results}
The paper is organized as follows.  In Section \ref{sec:prelim} we review necessary definitions and results for nonsingular bilinear maps and Grassmann manifolds, each of which play an integral role in the study of skew fibrations.  In Section \ref{sec:fibrations} we initiate a thorough study of the topology and geometry of skew and \emph{nondegenerate} (``first-order skew'') fibrations.   We introduce a special local vector bundle structure which is useful for explicit constructions for skew fibrations, and in Theorems \ref{thm:globalclass}, \ref{thm:skewclass}, and \ref{thm:nonclass} (which require too much notation to state here),  we prove classification results for skew and nondegenerate fibrations in the local and global settings.

In Section \ref{sec:contatinf} we study the key property of ``Continuity at Infinity'' which is satisfied by all locally skew fibrations of $\R^n$.  One important and surprising consequence is the following statement, which asserts that if parallelism occurs in a fibration of $\R^n$ by (not necessarily skew) copies of $\R^k$, then it must occur locally.

\begin{thm}
\label{thm:lg}
If a fibration of $\R^n$ by copies of $\R^k$ is locally skew, then it is globally skew.
\end{thm}

In Section \ref{sec:skewsphere} we turn our attention to understanding the precise correspondence between skew and spherical fibrations.  Recall that the central projection takes the open upper hemisphere of a fibered $S^n$ to a skew fibration of the tangent hyperplane $\R^n$ at the north pole; however, the fibers $S^k$ which lie completely on the equator $S^{n-1}$ (parallel to the tangent hyperplane $\R^n$) have no image in $\R^n$.   We will see that this information can still be recovered by the data of the skew fibration ``at infinity'' of $\R^n$. 

Conversely, the inverse central projection of a skew fibration of $\R^n$ covers the open upper hemisphere of $S^n$ by open great hemispheres.   Continuity at Infinity guarantees that completing all hemispheres to great spheres results in a continuous fibration of the open subset of $S^n$ consisting of the open upper and lower hemispheres  as well as the portion of the equator $S^{n-1}$ corresponding to the oriented directions which appear in fibers of the fibration.   We study when such a covering can be extended to a great sphere fibration of all of $S^n$.   As a simple example of when this completion process could fail, we may take a fibration of $\R^6$ by skew copies of $\R^2$.  Positioning $\R^6$ as the tangent hyperplane to $S^6$ and inverse central projecting yields a fibration of some open portion of $S^6$ by great copies of $S^2$ which can never be completed to a spherical fibration.  This phenomenon also exists in dimensions for which both fibrations exist.  In particular, the space of skew fibrations is strictly larger than the space of spherical fibrations.

The methods developed to study the relationship between skew and sphere fibrations also yield the following:

\begin{thm}
\label{thm:tf}
Given a fibration of $\R^n$ by skew oriented affine copies of $\R^k$, there exists a copy of $\R^{k+1}$ transverse to all fibers.
\end{thm}

When applied to the special case $n = 2k+1$, which is possible if and only if $k \in \left\{ 1, 3, 7 \right\}$ (see Section \ref{sec:admissible}), this result states that any skew fibration may be completely described by a map sending a point in some fixed transverse $\R^{k+1}$ to the unique affine $k$-plane through that point.   Combined with Theorems \ref{thm:skewclass} and \ref{thm:nonclass}, this allows for a complete classification of skew and nondegenerate fibrations of $\R^{2k+1}$ by $\R^k$ as those maps $\R^{k+1} \to \Hom(\R^k,\R^{k+1})$ satisfying certain nondegeneracy properties.

Another important consequence of Theorem \ref{thm:tf} is the fact that the base space of a fibration of $\R^{2k+1}$ by skew oriented affine $k$-planes must be homeomorphic to $\R^{k+1}$.  Thus the base spaces of any two such fibrations are homeomorphic, and such fibrations are topologically trivial, so we obtain the following result.

\begin{cor}
\label{cor:quesa}
Given a continuous (resp.\ smooth) fibration of $\R^{2k+1}$ by skew oriented affine copies of $\R^k$, there exists a fiber-preserving homeomorphism (resp.\ diffeomorphism) of $\R^{2k+1}$ taking the fibration to a Hopf fibration.
\end{cor}

This result gives a positive answer to Question \ref{ques:main} (a) for skew fibrations of $\R^7$ by $\R^3$ and $\R^{15}$ by $\R^7$.  We hope that the techniques developed in this paper could lead to answers for parts (b) and (c).

We continue studying the relationship between skew and sphere fibrations, and in Theorem \ref{thm:equivalent} we give several equivalent conditions for checking when a skew fibration completes to a sphere fibration.

%\begin{namedthm}[\ref{thm:equivalent}] A fibration of $\R^{2k+1}$ by skew oriented affine copies of $\R^k$ corresponds to a fibration of $S^{2k+1}$ by great spheres $S^k$ if and only if there is a unique linear subspace $\R^{k+1}$ transverse to all fibers.
%\end{namedthm}

Finally,  in Section \ref{sec:applications}, we use the machinery developed above to state and prove a number of results, focusing mostly on line fibrations and great circle fibrations.  We discuss the specific correspondence between these objects and how a great circle fibration can be constructed using a map $\R^{2m} \to \R^{2m}$ whose differential has no real eigenvalues.   The results we state are inspired by both classical and new results in the theory of sphere fibrations.  We see that results for skew fibrations are frequently more explicit and easier than their counterparts for sphere fibrations, and we discuss how the machinery above can be used to pass results about skew fibrations to those for sphere fibrations.  Two example results from Section \ref{sec:applications} are as follows:

\begin{thm}
\label{thm:germ}
Every germ of a smooth nondegenerate fibration of $\R^n$ by $\R^k$ extends to such a fibration of all of $\R^n$.
\end{thm}

\begin{thm}
\label{thm:contact}
For odd $n \geq 5$, there exists a smooth nondegenerate line fibration of $\R^n$ such that the orthogonal hyperplane distribution is not a contact structure.
\end{thm}

For comparison, it is known that the plane distribution orthogonal to a great circle fibration of $S^3$ \cite{Gluck} or to a smooth nondegenerate line fibration of $\R^3$ \cite{HarrisonBLMS} is a tight contact structure. 

In a recent paper of Cahn, Gluck, and Nuchi \cite{CahnGluckNuchi}, Question \ref{ques:main}(c) is labeled ``the best unsolved problem for great circle fibrations of spheres.''   Throughout Sections \ref{sec:applications} and \ref{sec:conclusion} we collect partial results and discuss possible methods for approaching Question \ref{ques:main}(c) for great circle fibrations and for line fibrations.

\section{Preliminaries}
\label{sec:prelim}

Here we discuss several relevant preliminary items.

\subsection{Admissible dimensions for skew fibrations}
\label{sec:admissible}
A fibration of $\R^n$ by oriented affine $k$-planes is called \emph{skew} if no two fibers contain parallel lines.  We have seen that skew fibrations arise via central projection of great sphere fibrations, but there exist skew fibrations which do not correspond via central projection to any great sphere fibration (see \cite{HarrisonBLMS} for an explicit example in $\R^3$).  In fact, fibrations of $\R^n$ by skew affine copies of $\R^k$ exist for many pairs of dimensions which do not admit spherical fibrations. 

In \cite{OvsienkoTabachnikov}, Ovsienko and Tabachnikov studied the question: for which pairs $k$ and $n$ does there exist a fibration of $\R^n$ by pairwise skew, oriented, affine copies of $\R^k$?  They found that a surprising number theoretic condition governs the existence of skew fibrations: a fibration of $\R^n$ by skew copies of $\R^k$ exists if and only if $k \leq \rho(n-k) - 1$, where $\rho$ is the \emph{Hurwitz-Radon} function, defined as follows: decompose $q$ as the product of an odd number and $2^{4a+b}$ for $0 \leq b \leq 3$; then $\rho(q) = 2^b + 8a$.   One important consequence is that a fibration of $\R^{2k+1}$ by skew copies of $\R^k$ exists if and only if $k \in \left\{ 1, 3, 7 \right\}$.  We offer two tables to help convey the existence dimensions more concretely.

\begin{figure}[ht]
\[
\begin{array}{c||c|c|c|c|c|c|c|c|c|c|c|c|c|c|c|c|c|c|c|c|c|c}
n & 3 & 4 & 5 & 6 & 7 & 8 & 9 & 10 & 11 & 12 & 13 & 14 & 15 & 16 & 17 & 18 & 19 & 20 & 21 & 22 & 23 & 24  \\
\hline \hline
k & 1 & & 1 & 2 & 3 & & 1 & 2 & 3 & 4 & 5 & 6 & 7 & & 1 & 2 & 3 & 4 & 5 & 6 & 7 & 8 \\
& & & & & 1 & & & & 1 & & 1 & 2 & 3 & & & & 1 & & 1 & 2 & 3 \\
& & & & & & & & & & & & & 1 & & & & & & & & 1 
\end{array}
\]
\caption{Pairs of dimensions $(k,n)$ for which skew fibrations exist}

\end{figure}

\begin{figure}[ht]
\[
\begin{array}{c||c|c|c|c|c|c|c|c|c|c|c|c|c|c}
k & 1 & 2 & 3 & 4 & 5 & 6 & 7 & 8 & 9 & 10 & 11 & 12 & 13 & 14  \\
\hline \hline
a_k & 2 & 4 & 4 & 8 & 8 & 8 & 8 & 16 & 32 & 64 & 64 & 128 & 128 & 128
\end{array}
\]
\caption{A skew fibration exists for $(k,n)$ if and only if $n = a_km + k$ for some $m \in \N$}
\end{figure}

The Hurwitz-Radon function was independently developed by Hurwitz and Radon a century ago in their studies of square identities (\cite{Hurwitz, Radon}).  Since then, the Hurwitz-Radon function has made prominent appearances in topology.  Most notably, in Adams' development of topological K-theory \cite{Adams}, he showed that the inequality $k \leq \rho(n-k) - 1$ holds if and only if there exist $k$ linearly independent vector fields on $S^{n-k-1}$.  In fact, the following statements are all equivalent:

\begin{compactitem}
\item $k \leq \rho(n-k) - 1$ 
\item there exist $k$ linearly independent vector fields on $S^{n-k-1}$
\item there exists a nonsingular bilinear map (see Section \ref{sec:nonsingular} below) $\R^{n-k} \times \R^{k+1} \to \R^{n-k}$ 
\item there exists a fibration of $\R^n$ by pairwise skew, oriented, affine copies of $\R^k$ 
\item there exists a $(k+1)$-dimensional linear subspace of the space of $(n-k) \times (n-k)$ matrices such that every nonzero matrix is invertible (this formulation, called Property P, was studied by Adams, Lax, and Phillips; see \cite{AdamsLaxPhillips,AdamsLaxPhillips2})
\end{compactitem}

In \cite{OvsienkoTabachnikov2}, Ovsienko and Tabachnikov gave a compelling expository account of the history surrounding the above ideas.  There they include an argument that a skew fibration exists if and only if $k \leq \rho(n-k) - 1$, assuming the first three equivalences above.  It is not difficult to show that a skew fibration induces $k$ linearly independent vector fields on $S^{n-k-1}$, and conversely, skew fibrations can be explicitly constructed using nonsingular bilinear maps (see also Lemma \ref{lem:nonnon} and Example \ref{ex:bilinear}). 

\subsection{Nonsingular bilinear maps}  \label{sec:nonsingular} A bilinear map $B \colon \R^q \times \R^p \to \R^d$ is called \emph{nonsingular} if $B(x,y) = 0$ implies $x=0$ or $y=0$.  In general, the minimum dimension $d$ such that there exists a nonsingular bilinear map $\R^q \times \R^p \to \R^d$ is unknown.   Nonsingular bilinear maps have been studied, in part, due to their relationships to immersions and embeddings of projective spaces \cite{James4}, totally nonparallel immersions \cite{HarrisonAdv}, the coindex of embedding spaces \cite{FrickHarrison}, and coupled embeddability \cite{FrickHarrison2}.  Most notably, they were studied in a series of articles by K.Y.\ Lam (e.g.\  \cite{Lam1, Lam6, Lam4, Lam3, Lam2, Lam5}).

We will see that nonsingular bilinear maps are essential not only for explicit, natural constructions of skew fibrations, but also for the classification of nondegenerate fibrations.

\subsection{The oriented (affine) Grassmann manifolds}
\label{sec:grassmann}

The fibers of a skew fibration are affine oriented $k$-planes in $\R^n$, and so the Grassmann and affine Grassmann manifolds play a large role in the study of skew fibrations.  Here we briefly review important aspects of their topology.

The oriented Grassmann $\grass$ is a manifold whose elements are oriented, linear, $k$-dimensional planes in $\R^n$, endowed with the smooth structure obtained by parametrizing a neighborhood of $u \in \grass$ by a neighborhood of $0 \in \Hom(u,u^\perp)$.   Specifically, the parametrization $\varphi \colon \Hom(u,u^\perp) \to \grass$ associates to a linear map $B$ the $k$-plane obtained as the graph of $B$; for example, $u$ itself is the graph of the zero map. 

The oriented affine Grassmann $\affgrass$ consists of oriented, affine, $k$-dimensional planes in $\R^n$.  An oriented affine $k$-plane $P \in \affgrass$ may be conveniently written as the pair $(u,v)$, where $u$ is the element of $\grass$ obtained by parallel translating $P$ to the origin, and $v$ is the point on $P$ which is nearest to the origin of $\R^n$.  Observe that $v$ necessarily lies on the copy of $\R^{n-k}$ which contains the origin of $\R^n$ and is orthogonal to $P$.   This establishes $\affgrass$ as the total space of the canonical $\R^{n-k}$-bundle $\rho \colon \affgrass \to \grass$, given by the projection $\rho(u,v) = u$.

There is an embedding $\alpha \colon \affgrass \to \grassy$ given as follows: consider $P \in \affgrass$ as an affine $k$-plane in $\R^{n} \times \left\{ 1\right\} \subset \R^{n+1}$, and map it to the $(k+1)$-plane induced by linearizing $P$.  Said differently,  $\alpha(P)$ is the $(k+1)$-plane in $\R^{n+1} = \R^n \times \R$  whose intersection with the hyperplane $\R^n \times \left\{ 1 \right\}$  yields $P$.  The embedding $\alpha$ will be used frequently when studying the correspondence between skew and sphere fibrations, since the inverse central projection of the tangent hyperplane $\R^n$ at the north pole of $S^n \subset \R^{n+1}$ sends an affine $k$-plane $P \subset \R^n$ to the great $k$-sphere $\alpha(P) \cap S^n \subset S^n$.

\begin{definition} Given $u \in \grass$, the \emph{bad cone} $\bee_u \subset \grass$ is the set consisting of $k$-planes $u'$ such that $u$ and $u'$ intersect nontrivially.
\end{definition}

\begin{remark} In the context of this article, the inequality $n > 2k$ always holds, so $\bee_u$ is a proper subset of $\grass$, and its complement is an open submanifold.   In particular $\bee_u$ is stratified according to the dimension of intersection with $u$.
\end{remark}

An important feature of the bad cone is that the tangent space to $\bee_u$ at $u$ is $\bee_u$ itself.   A careful argument is provided in \cite{GluckWarnerYang}; we give an outline here.  Using the parametrization $\varphi$ described above, observe that an element $u' \in \grass$ near $u$ lies in $\bee_u$ if and only if the linear map $\varphi^{-1}(u')$ is not injective.  Now we may identify the tangent space $T_0 \Hom(u,u^\perp)$ with $\Hom(u,u^\perp)$ itself to see that vectors tangent to the bad cone can be identified with elements of the bad cone.

The bad cone plays an essential role in the study of great sphere fibrations.  In particular, Gluck, Warner, and Yang showed that a smooth, closed, connected, $(n-k)$-dimensional submanifold $M$ of $\grassy$ is the base space of a smooth fibration of $S^n$ by great $k$-spheres if and only if $M$ is transverse to the bad cone $\bee_u$ for each $u \in M$ (\cite{GluckWarnerYang}, see also \cite[Proposition 1]{CahnGluckNuchi}).

\begin{remark} Here, and in the remainder of the article, we use the convention that two subspaces are \emph{transverse} at an intersection point $x$ if their tangent spaces at $x$ only intersect trivially.  In particular, there is no requirement that the tangent spaces span the ambient tangent space.
\end{remark}

For frequent future use, it is convenient to formalize the relationship between nonsingular bilinear maps and transversality to the bad cone. 

\begin{lem}
\label{lem:transversenonsingular}
Let $U$ be a smooth $r$-dimensional submanifold of $\grass$, given in some neighborhood of $u \in U$ by a smooth embedding $f \colon \R^r \to \grass$ with $f(0) = u$.   Then $U$ is transverse to $\bee_u$ at $u$ if and only if $d(\varphi^{-1} \circ f)_0 \colon T_0\R^r \to T_0 \Hom(u,u^\perp)$, considered as a map $\R^r \times \R^k \to \R^{n-k}$, is a nonsingular bilinear map. 
\end{lem}

\begin{proof} Due to the identification of tangent spaces discussed above, transversality of $U$ to $\bee_u$ at $u$ is equivalent to transversality of $\varphi^{-1}(U)$ to the variety of noninjective linear maps in $\Hom(u,u^\perp)$ at $0$.  This occurs if and only if the image of $d(\varphi^{-1} \circ f)_0$ appears as an $r$-dimensional linear subspace of $\Hom(u,u^\perp)$ transverse to the variety of noninjective linear maps; that is, the nonzero linear maps in the image all have full rank.  This is equivalent to nonsingularity of the corresponding map  $\R^r \times \R^k \to \R^{n-k}$. 
\end{proof}

A version of this correspondence was studied for great sphere fibrations in certain special cases, but never in the language of nonsingular bilinear maps (the objects are the same, but the terminology is different).  Gluck, Warner, and Yang \cite{GluckWarnerYang} studied great sphere fibrations of $S^{2k+1}$ by $S^k$ and their correspondence to regular algebras $\R^{k+1} \times \R^{k+1} \to \R^{k+1}$ (the Hopf fibrations correspond to the division algebras $\C, \Ham,$ and $\Oct$).  On the other hand, great circle fibrations of $S^{2m+1}$ correspond to linear maps $\R^{2m} \to \R^{2m}$ with no real eigenvalues; see \cite{McKay}, \cite{CahnGluckNuchi}.  Linear maps with no real eigenvalues will also make an appearance in the study of skew line fibrations; in particular see Corollary \ref{cor:noreal} and Section \ref{sec:gcnon}.

We conclude this section with the following observation, recalling the embedding $\alpha \colon \affgrass \to \grassy$ described above.

\begin{lem}
\label{lem:affinebadcone}
Two affine $k$-planes $P, Q \subset \R^n$ are skew if and only if $\alpha(Q) \notin \bee_{\alpha(P)}$.
\end{lem}

\begin{proof} The affine planes $P$ and $Q$ intersect at $p \in \R^n$ if and only if $\alpha(P)$ and $\alpha(Q)$ both contain the line spanned by $(p,1)$.  The affine planes $P$ and $Q$ contain the parallel line spanned by vector $w \in \R^n$ if and only if $\alpha(P)$ and $\alpha(Q)$ both contain the line spanned by $(w,0)$.
\end{proof}

\section{Topology and geometry of fibrations of $\R^n$}
\label{sec:fibrations}

Here we introduce and discuss our main objects of study: \emph{skew} and \emph{nondegenerate} fibrations of $\R^n$.   All fibrations here are vector bundles, whose fibers are affine, oriented $k$-planes in $\R^n$.  Viewing these objects from a local perspective, we discuss the vector bundle structure, and we establish some basic topological and geometric properties of skew/nondegenerate fibrations.  We introduce notation which will be used throughout the article.

%Intuitively we think of such an object as a continuous or smooth covering of $\R^n$ by pairwise skew, affine copies of $\R^p$.  In this section we see how this intuitive idea translates to the language of fibrations and vector bundles. 
%We describe these objects from both a local and global perspective, in both the topological and smooth categories.
%The local perspective allows us to work explicitly.

\subsection{Fibrations of $\R^n$: the local structure}

\label{sec:generalfibrations}

Consider a fibration of $\R^n$ by affine, oriented $k$-planes.  At the moment we do not assume that the fibers are pairwise skew, so the contents of this subsection apply equally well to the skew line fibration of Figure \ref{fig:skewlines} or to a fibration of $\R^n$ by parallel $k$-planes.

Each fiber $P$ is an element of the oriented affine Grassmann manifold $\affgrass$, and hence corresponds uniquely to a pair $(u,v)$ as follows: $v$ represents the point on $P$ nearest to the origin of $\R^n$, and $u$ represents the point in the oriented (linear) Grassmann manifold $\grass$ obtained by parallel translating $P$ from $v$ to the origin.  The fibration of $\R^n$ is the map which assigns to each $x \in \R^n$ the unique fiber $(u,v)$ through $x$.   We denote this $k$-plane bundle by $p \colon x \mapsto v$.

We can view the product structure of this vector bundle by looking at a neighborhood of a fiber $P = (u,v)$.  Consider the copy of $\R^q$, $q=n-k$, through the origin of $\R^n$ and orthogonal to $P$, so that $v = P \cap \R^q$.  Every fiber near $P$ is the graph of an affine map $P = \R^k \to \R^q \colon t \mapsto B(y)t + y$, where $y$ is the coordinate in $\R^q$ and $B(y) : \R^k \to \R^q$ is a linear map defined for $y$ sufficiently close to $v$.

Said differently, there is a map $B$ defined in a neighborhood $E \subset \R^q$ of $v$ and taking values in $\Hom(\R^k,\R^q)$, such that for a fixed $y$ in the neighborhood, the graph of the map $t \mapsto B(y)t + y$ is precisely the fiber through $y$.  By construction, $B(v) = 0$.  The induced map $E \times \R^k \to \R^n \colon (y,t) \mapsto (t,B(y)t + y)$ is a local trivialization of the vector bundle $p$, so $B$ is as smooth as the fibration.

%\begin{figure}[h]
%\centerline{
%\includegraphics[width=4in]{localfib.pdf}
%}
%\caption{Local depiction of a fibration}
%\label{fig:graph}
%\end{figure}

The restriction $p \big|_E \colon E \to p(\R^n)$ maps $y \in E$ to the point on the fiber through $y$ which is nearest to the origin.   This is a continuous injective map from the $q$-dimensional open set $E \subset \R^q$ to $p(\R^n)$, and hence is a homeomorphism onto its image $p(E)$.  This perspective is useful in that it allows us to visualize the base space $p(\R^n)$ of the fibration embedded as a $q$-dimensional submanifold of $\R^n$, consisting of all points in $\R^n$ which occur as ``nearest point to the origin'' for the fibration.

\begin{lem}
\label{lem:opencontractible}
The base space $p(\R^n)$ of a fibration of $\R^n$ by affine, oriented $k$-planes is a connected, contractible, $q$-dimensional submanifold of $\R^n$.
\end{lem}

\begin{proof} We have established that $p(\R^n)$ is a $q$-dimensional submanifold, which must be connected, as the image of $p$.   Now we have a long exact sequence of homotopy groups:
\[\cdots \rightarrow \pi_{n+1}(\R^n) \xrightarrow{p_*} \pi_{n+1}(p(\R^n)) \rightarrow \pi_n(\R^k) \rightarrow \cdots \rightarrow \pi_0(\R^k) \rightarrow \pi_0(\R^n) \rightarrow 0.\]
It follows that $p$ induces isomorphisms $\pi_m(\R^n) \approx \pi_m(p(\R^n))$ for all $m>0$, so by Whitehead's Theorem, $p$ is a homotopy equivalence and $p(\R^n)$ is contractible.  \end{proof}

\subsection{Skew fibrations of $\R^n$}

A fibration of $\R^n$ by affine, oriented $k$-planes is called \emph{skew} if no two fibers contain parallel lines.   For a skew fibration, the map $\pi \colon \R^n \to \grass$, which sends $x \in \R^n$ to the unique linear oriented $k$-plane through $x$, is injective.  This leads to a more useful manifestation of the base space as the subset $U \subset \grass$ consisting of $k$-planes which appear as fibers of the fibration.

Let $\xi_U$ be the total space of the restriction of the canonical affine Grassmann bundle (see Section \ref{sec:grassmann}) to $U \subset \grass$.   A skew fibration $\pi \colon \R^n \to U$ contains the data of a section $v \colon U \to \xi_U \colon u \mapsto (u,v(u))$.  Let $M \coloneqq v(U) \subset \affgrass$ be the collection of oriented affine planes which appear in the fibration.  As seen in the general case (for fibrations which are not necessarily skew), the section $v$ determines an embedding of $M$ in $\R^n$, as the set of points which occur as ``nearest point to the origin.''  Then $v : U \to v(U) \subset \R^n$ is a continuous injective map to a $q$-dimensional manifold, and therefore gives a homeomorpshism $U \to v(U)$.  This allows us to think of either $U$ or $v(U)$ as the base space of a skew fibration and interchange these perspectives.  Note that Lemma \ref{lem:opencontractible} establishes $U$ as a connected, contractible, $q$-dimensional submanifold of $\grass$.

\begin{lem}
\label{lem:closed}
If $M \subset \affgrass$ corresponds to a skew fibration, then $M$ is topologically closed in $\affgrass$.  Moreover, if $(u_n,v_n) \in M$ is a sequence with no accumulation point in $M$, $|v_n| \to \infty$.
\end{lem}

\begin{proof}  Let $P_n = (u_n,v_n) \in M$ be a sequence converging to a point $P = (u',v') \in \affgrass$.  Then $v_n = v(u_n) \in \R^n$ is the point nearest to the origin on $P_n$.  Now $v_n \to v'$, so continuity of $\pi$ implies the convergence $u_n = \pi(v_n) \to \pi(v')$.  Thus $u' = \pi(v')$.  Geometrically, this means that the plane through $v'$ is (the translation of) $u'$, so $(u',v')$ is contained in $M$, hence $M$ is closed.

Consider again a sequence $(u_n,v_n)$ in $M$.  We prove the contrapositive of the latter statement.  If the sequence of distances $|v_n|$ does not approach $\infty$, then there is a bounded subsequence, and hence a convergent subsequence $|v_{n_j}|$.  Thus $(u_{n_j},v_{n_j})$ is contained in a compact subset of $\affgrass$, and so it has an accumulation point, which is contained in $M$ by closure.  \end{proof}

%\begin{example}
%\label{ex:hopfvec} As depicted in Figure \ref{fig:skewlines}, the central projection of the standard Hopf fibration of $S^3$ by oriented great circles gives a fibration of $\R^3$ by skew oriented lines.  In this case, the subset $U \subset S^2$ is the open upper hemisphere and $AG_1(3) \simeq TS^2$, so that the section $v$ is a tangent vector field on $U$.  Explicitly, $v$ sends a point $(\sqrt{u_1^2+u_2^2+1})^{-1} \cdot (u_1,u_2,1) \in U$ to the tangent vector $(-u_2,u_1,0)$.  %This example may be manipulated to see the necessity of the second bullet point of Theorem \ref{thm:contfibvectorfield}. Restricting the domain of $v$ to a smaller spherical cap corresponds to removing some lines from the fibration, hence the resulting collection of lines are skew but do not cover $\R^3$.
%\end{example}

Consider a fibration of $\R^n$ by oriented, affine $k$-planes; at the moment we do not assume that the fibers are skew.  Let $z \in \R^n$, and let $P = (u,v)$ be the affine fiber through $z$.  As discussed in Section \ref{sec:generalfibrations}, there is a map $B$ defined in a neighborhood $E$ of $z$ in $u^\perp$ and taking values in $\Hom(u,u^\perp) = \Hom(\R^k,\R^q)$, such that for a fixed $y$ in the neighborhood, the graph of the map $t \mapsto y + B(y)t$ is precisely the fiber through $y$.  

Given a map $B \colon E \to \Hom(\R^k, \R^q)$ let $A \colon E \to \Hom(\R^{k+1}, \R^q)$ be the map defined by $A(y)(t,\lambda) = B(y)t + \lambda y$.   In coordinates, if $B(y)$ is represented by a $q \times k$ matrix, the matrix for $A(y)$ is obtained by appending the column vector $y$ to $B(y)$.

\begin{lem}
\label{lem:toplocal}
Suppose that $A \colon E \to \Hom(\R^{k+1}, \R^q)$ corresponds to a fibration of (some open subset of) $\R^n$ in the manner described above.  Then the fibration is skew if and only if for every distinct $x,y \in E$, $\Ker(A(x) - A(y)) = 0$.
\end{lem}

\begin{proof}
A nonzero vector $(t,0) \in \R^k \times \R$ is in $\Ker(A(x)-A(y))$ if and only if $B(x)t = B(y)t$, which occurs if and only if the fibers through $x$ and $y$ contain a parallel line.  Given nonzero $t$, $(t,1)$ is in $\Ker(A(x)-A(y))$ if and only if $x-y$ is in the linear span of $B(x) - B(y)$, which occurs if and only if $x$ and $y$ intersect.
\end{proof}

\begin{remark} The geometric interpretation of the statement $\Ker(A(x) - A(y)) = 0$ can be understood in the context of Section \ref{sec:grassmann}, since it holds if and only if the $k$-planes $P_x$ and $P_y$ through $x$ and $y$ are skew, which occurs if and only if the $(k+1)$-plane $\alpha(P_y)$ does not lie in the bad cone $\bee_{\alpha(P_x)} \subset \grassy$.  
\end{remark}

\subsection{Nondegenerate fibrations of $\R^n$}

\label{sec:nondegenerate}

We now turn our attention to smooth fibrations of $\R^n$.  Our first task is to develop a first-order notion of skewness.

A smooth fibration of odd-dimensional space $\R^n$ by oriented lines may be thought of as a unit vector field $V$ on $\R^n$, for which all of the integral curves are lines.   Equivalently, $\nabla V$ vanishes in the direction of $V$: $\nabla_V V \equiv 0$.  Following Salvai \cite{Salvai}, we say that such a fibration is \emph{nondegenerate} if $\nabla V$ vanishes \emph{only} in the direction of $V$.   The complete opposite situation,  in which $\nabla V$ is identically $0$, corresponds to a fibration of $\R^n$ by parallel lines. 

Any nondegenerate fibration is locally skew.  Indeed, let $z \in \R^n$, and let $P$ be the fiber containing $z$.  The nondegeneracy condition is equivalent to the statement that $V \big|_{P^\perp} \colon P^{\perp} \to S^{n-1}$ is an immersion at $z$.  Therefore $V \big|_{P^\perp}$ is locally injective, so no two fibers near $P$ are parallel. 

To extend the notion of nondegeneracy to fibrations with higher-dimensional fibers, we consider a smooth map $\pi : \R^n \to \grass$, for which all $k$-dimensional integral submanifolds are $k$-planes.  Equivalently, $d \pi_x$, considered as a map on $T_x \R^n = \R^n$, vanishes on $\R^k = \pi(x)$.   A reasonable generalization of nondegeneracy is the condition that $d \pi_x$ vanishes \emph{only} on $\pi(x)$; equivalently,  $\rank(d\pi_x) = q = n-k$ for all $x$.  However, while this condition guarantees that no two fibers are parallel as $k$-planes, it does not guarantee that no two fibers contain parallel lines, so this condition does not adequately capture the notion of skewness.

We recall from Section \ref{sec:grassmann} that the \emph{bad cone} $\bee_u$ at $u \in \grass$ is the set of $u' \in \grass$ such that the $k$-planes $u$ and $u'$ intersect in at least a line.

\begin{definition}
A smooth fibration $\pi \colon \R^n \to U \subset \grass$ is \emph{nondegenerate} if for all $x \in \R^n$, $\pi\big|_{\pi(x)^\perp}$ is an immersion transverse to the bad cone $\bee_{\pi(x)}$.
\end{definition}

Note that if $\pi$ is nondegenerate at a point $x$, then nondegeneracy holds in a neighborhood $E\subset P^\perp$ of $x$.  Moreover, this neighborhood may be chosen so that $\pi \big|_{P^\perp}$ takes $E$ injectively to the complement of the union (over $y \in E$) of the bad cones $\bee_{\pi(y)} \subset \grass$.  Therefore the fibers at $y_1, y_2 \in E$ do not share parallel lines, and hence any nondegenerate fibration is locally skew. 

%\begin{thm}
%A locally skew fibration is globally skew.  In particular, any nondegenerate fibration $\R^p \to \R^n \xrightarrow{\pi} U$ is skew.
%\end{thm}

The nondegeneracy condition may therefore be thought of as ``first-order skewness,'' much like the immersion condition may be thought of as ``first-order embeddedness.''  However, the remarkable feature of locally skew fibrations known as Continuity at Infinity (see Section \ref{sec:contatinf}) ensures that a locally skew fibration is globally skew, and so nondegeneracy actually yields global skewness. 

Another justification of the definition of nondegeneracy will be seen in Section \ref{sec:skewsphere}, where we see that any smooth fibration of $\R^n$ which corresponds via central projection to a smooth great sphere fibration (with no first-order condition!) must be nondegenerate.

We will now prove a version of Lemma \ref{lem:toplocal} for nondegenerate fibrations.   In light of Section \ref{sec:grassmann}, especially Lemma \ref{lem:transversenonsingular}, it should not be surprising that nondegeneracy of $\pi$ at a point $x$, which concerns transversality to the bad cone at $\pi(x)$,  corresponds to nonsingularity of an associated bilinear map.  In fact, since $B = \varphi \circ \pi \big|_E$, the associated nonsingular bilinear map is precisely $dB_x$.  It is somewhat surprising that nondegeneracy of $\pi$ actually corresponds to something stronger: nonsingularity of $dA_x$.  This is formalized as follows.

\begin{lem}
\label{lem:nonnon}
Suppose that $A \colon E \to \Hom(\R^{k+1}, \R^q)$ corresponds to a fibration of (some open subset $T$ of) $\R^n$ in the manner described above.  Then the fibration is nondegenerate if and only if for every $y \in E$, $dA_y \colon \R^q \times \R^{k+1} \to \R^q$ is a nonsingular bilinear map.
\end{lem}

\begin{remark} In \cite{OvsienkoTabachnikov}, Ovsienko and Tabachnikov showed that the existence of a fibration of $\R^n$ by skew affine copies of $\R^k$ implies $k + 1 \leq \rho(q)$ by constructing $k$ linearly independent tangent vector fields on $S^{q-1}$ as follows: Consider a fiber $P$ and project those fibers which pass through a small sphere $S \subset E \subset P^\perp$ to the tangent spaces of $S$; it is straightforward to check that the skew condition implies linear independence.  Lemma \ref{lem:nonnon} can be viewed as an infinitesimal version of this argument, which applies to smooth nondegenerate fibrations, and which uses the nondegeneracy to construct a nonsingular bilinear map $\R^q \times \R^{k+1} \to \R^q$, yielding the same obstruction.
\end{remark}

\begin{proof}[Proof of Lemma \ref{lem:nonnon}]
Consider a smooth fibration $\pi \colon T \to \grass$.   Let $z \in T$ and let $(u,v)$ be the fiber through $z$.  For simplicitly we will assume that $z$ is the origin in $\R^n$, and we will consider the orthogonal decomposition $\R^n = u \times u^\perp \approx \R^k \times \R^q$.  The fibration $\pi$ induces maps $B \colon E \subset \R^q \to \Hom(\R^k,\R^q)$ and $A \colon E \subset \R^q \to \Hom(\R^{k+1},\R^q)$. Writing $\pi \big|_E = \varphi \circ B$, Lemma \ref{lem:transversenonsingular} tells us that
\begin{itemize}
\item nondegeneracy of $\pi$ at $y$ corresponds to nonsingularity of $dB_y$; that is, for nonzero $\xi \in \R^q$ and nonzero $t \in \R^k$, $0 \neq dB_y(\xi)(t) = dA_y(\xi)(t,0)$.
\end{itemize}
It remains to show that
\begin{itemize}
\item nondegeneracy of $\pi$ at other points in the fiber containing $y$ corresponds to the statement that $0 \neq dA_y(\xi)(t,1)$ for nonzero $\xi \in \R^q$.
\end{itemize}
Loosely speaking, the first item is a first-order condition representing nonparallelism, and the second item is a first-order condition representing the nonintersection property (compare with Lemmas \ref{lem:affinebadcone} and \ref{lem:toplocal}).

Now, given $t \in \R^k$, let $\R^q_{t} = \left\{t\right\} \times \R^q$ be the copy of $\R^q$ obtained by translating $u^\perp$ by $t$.  Define $f \colon E \to \R^q_t \subset \R^n$ by $f(y) = (t,A(y)(t,1)) = (t,B(y)t + y)$.  Geometrically, $f$ takes a point $y \in E$ to the intersection point of $\R^q_t$ with the fiber through $y$.   This intersection point is unique since any fiber which intersects $E$ transversely also intersects $\R^q_t$ transversely.   Thus $f$ is a bijection onto its image $E_t \coloneqq \R^{q}_t \cap T$. 

It is convenient to consider the following diagram
\[
\begin{tikzcd}[row sep=2.5em, column sep=1.25em]
    E \arrow[rr, "f"] \arrow[swap, dr, "\pi|_E"] & & E_t \arrow[dl, "\pi|_{E_t}"] \\
    & \grass \\[-3em]               
\end{tikzcd}
\]
which efficiently shows that the nondegeneracy of $\pi$, which implies that $\pi \big|_E$, and $\pi \big|_{E_t}$ are both diffeomorphisms onto the same image in $\grass$,  yields that $f$ is a diffeomorphism.  Thus for nonzero $\xi$, we have $0 \neq df_y(\xi) = dA_y(\xi)(t,1) \in T_{f(y)}\R^q_{t} = \R^q$.    Conversely, if $dA_y$ is nonsingular on $E$, then so is $dB_y$, and the argumentation above shows that this is equivalent to the statement that $\pi\big|_E$ is a diffeomorphism onto its image.   Moreover, nonsingularity of $A$ implies that $f$ is a diffeomorphism.  Thus $\pi\big|_{E_t} = \pi\big|_E \circ f^{-1}$ is a diffeomorphism whose image is transverse to the bad cone, and so $\pi$ is nondegenerate on $E_t$.
\end{proof}

\begin{remark} To further emphasize the split between nonparallelism and nonintersection, we note that if $\pi$ is only assumed to be smooth, each $f$ is still a diffeomorphism.  Indeed,  we may write $f = (p \big|_{E_t})^{-1} \circ (p \big|_E)$, where $p$ maps $y \in \R^n$ to the point on the fiber through $y$ which is nearest to the origin (see Section \ref{sec:generalfibrations}).  Thus even without the nondegeneracy assumption,  it still follows that $dA_y(\xi)(t,1)$ is nonzero for nonzero $\xi$. 
\end{remark}

\begin{cor}
\label{cor:noreal}
Suppose that $B \colon E \to \R^{n-1}$ is obtained from a smooth nondegenerate fibration of $\R^n$ by lines in the manner described above.  Then for every $y \in E$, $dB_y \colon \R^{n-1} \to \R^{n-1}$ has no real eigenvalues.
\end{cor}

\begin{proof}
The statement can be proven directly from nonsingularity of $A$ and the relationship between $A$ and $B$, but we prefer to be somewhat more explicit.   Recall that a smooth nondegenerate line fibration of $\R^n$ is given by a smooth unit vector field $V \colon \R^n \to S^{n-1}$, such that all integral curves are lines, and such that $\nabla V$ vanishes only in the direction of $V$.   Given $z \in \R^n$ and the usual neighborhood $E \subset u^\perp$, $V\big|_E$ is an immersion by nondegeneracy, so $B$ is a diffeomorphism onto its image (and hence zero is not an eigenvalue of $dB_y$).  Now by the same argument as given in Lemma \ref{lem:nonnon},  the map $f_t \colon E \to \R^{n-1}_t \subset \R^n \colon y \mapsto (t,tB(y) + y)$ is a diffeomorphism onto its image for fixed nonzero $t$.  Thus for nonzero $\nu$, $0 \neq d(f_t)_y(\nu) = tdB_y(\nu) + \nu$, and so $-\frac{1}{t}$ is not an eigenvalue of $dB_y$.
\end{proof}

\begin{remark} To continue the remark above: even without the nondegeneracy assumption, the fact that the oriented lines do not intersect implies that $dB_y$ has no \emph{nonzero} real eigenvalues.  See \cite{HarrisonAGT} for further discussion on line fibrations of $\R^3$ which are not necessarily skew, including the no nonzero real eigenvalue condition and the relationship with tight contact structures of $\R^3$.
\end{remark}

\subsection{A characterization of nondegenerate fibrations}

In this section we see precisely when a submanifold of $\affgrass$ corresponds to a smooth nondegenerate fibration of $\R^n$.

We use the notation $\rho \colon \affgrass \to \grass$ for the projection of the canonical bundle, and $\xi_U \to U$ refers to the bundle obtained by restricting $\rho$ to $\xi_U = \rho^{-1}(U)$ to some base space $U \subset \grass$.

\begin{thm} 
\label{thm:globalclass}
Let $M \subset \affgrass$ be a smooth $(n-k)$-dimensional submanifold.  The following are equivalent:
\begin{compactenum}[(a)]
\item The elements of $M$ are the fibers of a smooth nondegenerate fibration $\pi \colon \R^n \to U$.
\item $M$ is topologically closed, connected, and the image $\alpha(M)$ is transverse to every bad cone.
\item $M$ is the image of a smooth section $v \colon U \to \xi_U$ for some connected, contractible smooth submanifold $U \subset \grass$ such that $\alpha(M)$ is transverse to every bad cone and $u_n \to \partial U$ implies $|v(u_n)| \to \infty$. 
\end{compactenum}
\end{thm}

\begin{proof}
``(a) $\Rightarrow$ (c)'':  By Lemma \ref{lem:opencontractible}, the base space $U$ is a connected, contractible, $(n-k)$-dimensional submanifold of $\grass$.   By Corollary \ref{cor:globallyskew}, which will be shown following the Continuity at Infinity investigation of Section \ref{sec:contatinf}, any nondegenerate fibration is skew, and so by Lemma \ref{lem:closed}, the corresponding smooth section $v \colon U \to \xi_U$ satisfies the condition that $u_n \to \partial U $ implies $|v(u_n)| \to \infty$.  Now given $x \in \R^n$ and its containing plane $P = (u,v)$, the induced map $A \colon E \to \Hom(\R^{k+1},\R^q)$ satisfies the property that $dA_x$ is nonsingular (Lemma \ref{lem:nonnon}), which by Lemma \ref{lem:transversenonsingular} implies that $\alpha(M)$ is transverse to the bad cone at $\alpha(P)$.

``(c) $\Rightarrow$ (b)'':  The topological closedness follows from the final condition of (c).

``(b) $\Rightarrow$ (a)'':  The idea of the argument is as follows: we have shown in Lemma \ref{lem:nonnon} that the transversality to the bad cone corresponds to nondegeneracy; however, this is a local statement.  The key idea is that the additional conditions on $M$ allow for the application of a ``global inverse'' result of Hadamard (see \cite[Theorem 6.2.8]{KrantzParksBook} or \cite{Gordon}), which asserts that a smooth proper local diffeomorphism from a smooth connected manifold $M$ to $\R^n$ is a diffeomorphism; here \emph{proper} means that the preimage of every compact set is compact.

There is a $k$-plane bundle over $\grass$ for which the fiber over $u \in \grass$ is $u$ itself.   We will refer to the total space of this bundle as $F$.  Consider the pullback $\rho^*F$ and its restriction to $M$.  Define the map $f \colon \rho^*F\big|_M \to \R^n \colon (u,v,t) \mapsto v + t$, where $v$ represents the nearest point from $(u,v)$ to the origin of $\R^n$.   Geometrically, the image of $f$ consists of those points lying in the $k$-planes which are elements of $M$.

Transversality of $\alpha(M)$ to the bad cone at $\alpha(u,v)$ implies transversality of $\rho(M)$ to the bad cone at $\rho(u,v) = u$; in particular $\rho\big|_N$ maps a neighborhood $N \ni (u,v)$ diffeomorphically to a neighborhood $\rho(N)$ of $u$ in $\grass$, and $N$ is the image of a smooth section defined on $\rho(N)$.   The restriction of $f$ to $\rho^*F\big|_N$ parametrizes an open fibered neighborhood $T$ of the affine $k$-plane $(u,v)$ in $\R^n$.  Define $\pi \colon T \to N \colon v + t \mapsto v$.   Then $\pi$ is a smooth map on $T$ whose fibers $\pi^{-1}(v)$ are $k$-planes $(u,v)$ from $N$, and by transversality to the bad cone and Lemma \ref{lem:nonnon}, $\pi$ is a smooth nondegenerate fibration of $T$.    Therefore $f$ is an immersion, hence a local diffeomorphism.

We show that $f$ is proper.  Consider the preimage $f^{-1}(K)$ of a compact subset $K \subset \R^n$.  Let $(u_k,v_k,t_k)$ be a sequence in $f^{-1}(K)$, and let $x_k = f(u_k,v_k,t_k)$.   By compactness of $\grass$, $u_k$ admits a convergent subsequence to $u \in \grass$.  By the fact that $|v_k| \leq |x_k|$ is bounded by the maximum distance from the origin to $K$, $v_k$ admits a further convergent subsequence to $v$.  Finally, $|t_k| = |v_k - x_k|$ is bounded, and so $t_k$ admits a further convergent subsequence to $t$.  Since $M$ is topologically closed, $(u,v,t) \in M$.  Thus $f^{-1}(K)$ is compact, and hence $f$ is proper.

Therefore $f$ is a diffeomorphism, and the (local) map $\pi$ defined above gives a global smooth nondegenerate fibration of $\R^n$.
%CLOSEDNESS ARGUMENT SOMETHING LIKE: Consider the map $U \to \R : u \mapsto |v(u)|$, and let $m$ be the infimum of the image.  There is a sequence $u_n$ with $|v(u_n)|$ approaching $m$.  Choose a convergent subsequence of $u_n$; the limit $u$ must be in $U$ by the second bullet point.  Thus the infimum $m$ is a minimum, achieved at the point $u$.  Assume that $m$ is not $0$, so that $u$ does not contain the origin of $\R^n$.

%Consider in $\R^n$ the copy of $\R^q = u^\perp$ through the origin.  Let $V \subset U$ be a neighborhood of $u$, small enough so that for each $p$-plane $u' \in V$, $u'$ intersects $\R^q$ transversely.  The map $V \to \R^q$, which sends each $u'$ to its intersection point with $\R^q$, is continuous and injective, hence the image is an open subset of $\R^q$, by invariance of domain.  But in such a subset, there is a point neared than $v(u)$ to the origin, a contradiction.  Thus $m$ must be $0$, and so $u$ passes through the origin of $\R^n$.  For an arbitrary point in $x \in \R^n$, applying the same argument to the map $U \to \R : u \mapsto |x - v(u)|$ ensures that some plane passes through $x$. 

%Now the image of $f$ is open and closed, so it must be all of $\R^n$.  Since $f$ is a proper local diffeomorphism to $\R^n$, it must be a diffeomorphism by a theorem of Hadamard.
\end{proof}

\begin{remark} The above classification generalizes two prior results.  A geometric version of this classification was given by Salvai in the case of smooth nondegenerate line fibrations of $\R^3$ \cite{Salvai}: the classification was not given in the topological language of transversality to the bad cone in $AG_1(3)$, but rather, as definiteness of $M$ with respect to a certain pseudo-Riemannian metric on $AG_1(3) \sim TS^2$ which does not exist in the general case of $\affgrass$. In \cite{Harrison}, we stated a version of the above classification for skew fibrations, but there the transversality to the bad cone was replaced with the condition that the elements of $M$ are skew, and thus much of the difficulties in the above proof were avoided there.
\end{remark}

\begin{remark} For skew line fibrations of $\R^3$, the base space $U \subset G_1(3) \sim S^2$ is not only contractible, but convex in $S^2$ (see \cite{Salvai}, \cite{Harrison} for two different arguments) and homeomorphic to $\R^2$.   Learning more about the topology and geometry of the base space of a skew line fibration of $\R^n$ is one of the most important prerequisites for future advancements for line and great circle fibrations.  
\end{remark}

\begin{example} The image $M$ of the zero section $S^2 \to TS^2$ is a topologically closed, connected submanifold of $TS^2 \sim AG_1(3)$ for which $\rho(M) = S^2$ is transverse to the bad cone at every point, but the resulting collection of lines intersect at the origin.  This example highlights the difference between transversality of $\rho(M)$ to the bad cones in $\grass$ and transversality of $\alpha(M)$ to the bad cones in $G_{k+1}(n+1)$.
\end{example}

\subsection{When local becomes global: fibrations with a transverse $\R^q$}
\label{sec:global}

We have seen that a skew (or nondegenerate) fibration $\pi \colon \R^n \to \grass$ corresponds, in a neighborhood $E \subset P^\perp$ of a point $z$ on a fiber $P$, to a map $B \colon E \to \Hom(\R^k, \R^q)$.   The map $B$ is only defined locally, since a fiber far from $P$ may fail to be transverse to $P^\perp$.  However, if there happens to exist a copy of $\R^q$ which is transverse to all fibers from the fibration, then a map $B$ may be defined globally, and the entire data of the fibration may be recovered from the map $B$. 

We will see later in this section that every fibration of $\R^{2k+1}$ by skew affine copies of $\R^k$ (which occurs if and only if $k \in \left\{ 1, 3, 7, \right\}$) admits a transverse copy of $\R^{k+1} = \R^q$, so there exists a globally-defined map $B$.  In fact, we show that every fibration of $\R^n$ by skew affine copies of $\R^k$ admits a transverse $\R^{k+1}$ (here we say that two affine subspaces of $\R^n$ are transverse if their linear translates meet only in the origin; in particular their dimensions may not be complementary). 

Before initiating this lengthy argument, we first study the consequences of a globally-defined map $B$.   To this end, we define $\mathscr{F}$ as the space of all skew (or nondegenerate) fibrations of $\R^n$ by $\R^k$.  When studying skew fibrations, we consider $\mathscr{F}$ in the subspace topology inherited from $C^0(\R^n,\grass)$.  When studying nondegenerate fibrations, we consider $\mathscr{F}$ in the subspace topology inherited from $C^1(\R^n,\grass)$, though results in the $C^k$ or $C^\infty$ topologies will be the same.   It will be clear, either from context or from explicit statement, whether we refer to skew or nondegenerate fibrations.  Let $\mathscr{F}_{tr}$ represent the subspace of $\mathscr{F}$ consisting of those fibrations for which there exists a copy of $\R^q$ transverse to all fibers.

%A priori, this setup is only local: outside of a small neighborhood of $z$, a fiber may no longer be transverse to the copy of $\R^q$, so it is not the graph of a function as described above (see the fiber labeled $P$ in Figure \ref{fig:graph}).  If we were assured of the existence of a copy of $\R^q$ transverse to all fibers, then the setup above would be global, and the entire fibration could be defined in terms of a single map $B : \R^q \to \Hom(\R^p, \R^q)$.  All known examples of $(p,n)$-fibrations are global in this sense.  In particular, the $(p,n)$-fibrations constructed via Property P use \emph{linear} maps $B : \R^q \to \Hom(\R^p, \R^q)$; that is, $B \in \Hom(\R^q, \Hom(\R^p,\R^q))$.

%We will show the following in Section \ref{sec:transverse}:

%\begin{thm}
%Given a skew fibration $\R^p \to \R^n \xrightarrow{\pi} U$, there exists a copy of $\R^{p+1}$ transverse to all fibers.
%\end{thm}

%\begin{cor}
%Given a skew fibration $\R^p \to \R^{2p+1} \xrightarrow{\pi} U$, which requires $p \in \left\{ 1,3,7 \right\}$, there exists a copy of $\R^{n-p}$ transverse to all fibers.  That is, $\mathscr{F} = \mathscr{F}_{tr}$.
%\end{cor}

All known examples of skew/nondegenerate fibrations admit a transverse copy of $\R^q$, so it may be the case that $\mathscr{F} = \mathscr{F}_{tr}$ for every pair of dimensions. We do not have any tools to study fibrations in the (possibly empty) space $\mathscr{F} - \mathscr{F}_{tr}$.

Given a point $x \in \R^n$, let $d(x)$ represent the minimum distance from the origin to the fiber $P = (u,v)$ through $x$.  That is, $d(x) = |v|$.

\begin{thm}
\label{thm:skewclass}
A continuous map $B : \R^q \to \Hom(\R^k,\R^q)$ corresponds to a skew fibration, in the manner described above, if and only if
\begin{itemize}
\item for every pair of distinct points $x, y \in \R^k$, $\Ker(A(x) - A(y)) = 0$, and
\item if $x_n$ is a sequence of points in $\R^q$ with no accumulation points, then $d(x_n) \to \infty$.
\end{itemize}
Conversely, any skew fibration in $\mathscr{F}_{tr}$ can be described by such a $B$ on a suitable $\R^q \subset \R^n$. 
\end{thm}

\begin{proof}
Let $B \colon \R^q \to \Hom(\R^k, \R^q)$ be a continuous map which corresponds to a skew fibration $\pi \colon \R^n \to \grass$.  The first item is a consequence of  Lemma \ref{lem:toplocal}.  Let $x_n$ be a sequence in $\R^q$ and let $(u_n,v_n)$ be the fiber containing $x_n$, so that $d(x_n) = |v_n|$.  If $|v_n| \nrightarrow \infty$, then by Lemma \ref{lem:closed} there exists a convergent subsequence $(u_{n_j},v_{n_j})$ to some fiber $(u,v)$, which intersects $\R^q$ at some point $x$.   Since $(u_{n_j},v_{n_j}) \to (u,v)$, $x_{n_j} \to x$, so $x$ is an accumulation point of $x_n$.

Now let $B \colon \R^q \to \Hom(\R^k, \R^q)$ be a continuous map which satisfies both items.  The first item guarantees that the planes are skew.  We must show that the planes cover all of $\R^n$.   Let $\R^q_0 \subset \R^n$ be the domain of $B$.  As in the proof of Lemma \ref{lem:nonnon}, let $\R^q_{t} = \left\{t\right\} \times \R^q$ be the copy of $\R^q$ obtained by translating $\R^q_0$ by $t$.  The map $f \colon \R^q_0 \to \R^q_t \subset \R^n$, defined by $f(y) = (t,A(y)(t,1)) = (t,B(y)t + y)$, is a homeomorphism onto its image in $\R^{q}_{t}$.   Indeed, $f$ takes a point $y \in \R^q_0$ to the intersection point of the fiber through $y$ with $\R^q_t$, and hence is continuous and injective.  In particular, $f(\R^q_0)$ is open in $\R^q_t$. 

Now we show that $f(\R^q_0)$ is closed, since then it is equal to $\R^q_t$.  Let $z_n$ be a sequence in $f(\R^q_0)$ converging to a point $z$.  Let $y_n$ be the preimage of $z_n$ and $(u_n,v_n)$ the corresponding fiber.  Now $z_n$ lies on $(u_n,v_n)$ and $z_n \to z$,  so $|v_n|$ is a bounded sequence.  Thus by the second item, the sequence $y_n$ has a subsequence $y_{n_k}$ converging to $y$.  By continuity of $f$, $z_{n_k} = f(y_{n_k}) \to g(y)$.  But $z_{n_k}$ also approaches $z$, so $f(y) = z$.  Hence $f(\R^q_0)$ is closed.

Finally, we note that continuity of the map $F \colon \R^k \times \R^q \to \R^n \colon (t,y) \mapsto (t,A(y)(t,1))$ yields continuity of the induced map $\pi \colon \R^n \to \grass$, so that $B$ corresponds to a (continuous) fibration of $\R^n$.

The final statement of the theorem follows from the definition of $\mathscr{F}_{tr}$.
\end{proof}

\begin{thm}
\label{thm:nonclass}
A smooth map $B : \R^q \to \Hom(\R^k,\R^q)$ corresponds to a smooth nondegenerate fibration, in the manner described above, if and only if
\begin{itemize}
\item for every point $x \in \R^q$, $dA_x$ is a nonsingular bilinear map.
\item if $x_n$ is a sequence of points in $\R^q$ with no accumulation points, then $d(x_n) \to \infty$.
\end{itemize}
Conversely, any smooth nondegenerate fibration in $\mathscr{F}_{tr}$ can be described by such a $B$ on a suitable $\R^q \subset \R^n$. 
\end{thm}

\begin{proof}
With Lemma \ref{lem:nonnon} used in place of Lemma \ref{lem:toplocal}, the forward direction is identical to that in the proof of Theorem \ref{thm:skewclass}.  The reverse direction is almost identical, with one important difference.

Suppose that $B$ satisfies the two bullet points, and consider the continuous map $f \colon \R^q_0 \to \R^q_t \subset \R^n$ defined in the proof of Theorem \ref{thm:skewclass}.  There, we knew that $f$ was injective due to the global skewness assumed in the first bullet point of Theorem \ref{thm:skewclass}, and we applied Invariance of Domain to see that $f$ was a homeomorphism onto its image.  Here, we know instead that $f$ is a local diffeomorphism, but it is not obvious that $f$ is a global diffeomorphism.

Thus we apply the same ``global inverse'' result of Hadamard as in Theorem \ref{thm:globalclass}: here $f$ is a local diffeomorphism, which is proper by the second bullet point, and therefore is a diffeomorphism onto its image.
\end{proof}

\begin{remark}
In the $(1,3)$ case, it is possible to arrange such that $B(0)=0$.  Geometrically, this means that we choose the transverse plane so that it is orthogonal to some fiber $u$, and we label the intersection of $u$ and $u^\perp$ as the origin of the fibered $\R^3$.  In the $(3,7)$ and $(7,15)$ cases, this is too much to expect: the $k$-plane orthogonal to any transverse $q$-plane may not even appear as a fiber, as shown in the following example.
\end{remark}

%\note{Choose coordinates such that $B(0)=0$, even with no $\perp$ plane?}  Yes, at least up to homotopy.

\begin{example}
\label{ex:noorth}
The standard Hopf fibration of $S^7$ by $S^3$ is induced by quaternionic multiplication.  One can choose any copy of $\R^4$ in $\R^7$ to serve as the transverse plane and define $B \colon \R^4 \to \operatorname{Hom}(\R^4,\R^3)$ by $y \mapsto (iy,jy,ky)$.  Then $\det(A(y)-A(z)) = \det(y-z,i(y-z),j(y-z),k(y-z)) = |y-z|^4$.  Now modify by defining $y \mapsto (iy,jy,ky+x)$ for some fixed nonzero vector $x \in \R^4$.  This does not change the determinant, but here the $0$ linear map is not in the image, so the $3$-plane orthogonal to $\R^4$ is not a fiber of this skew fibration.
\end{example}

\begin{example}
\label{ex:bilinear}
Let $A \colon \R^q \times \R^{k+1} \to \R^q$ be a nonsingular bilinear map. We construct a nondegenerate fibration of $\R^{k+q}$ using $A$.   We may think of $A$ as a $q$-dimensional linear subspace $L$ in the set of $q \times (k+1)$ matrices such that every matrix besides the zero matrix has rank $k+1$.   The idea is to choose a column to serve as the coordinate $y \in \R^q$, and use the remaining $k$ columns for the matrix $B$.

More precisely, consider the linear map sending a matrix in $L$ to its last column.  This is a linear injection, and hence an isomorphism, $\R^q \to \R^q$, since if two matrices share the last column, their difference is not full rank, hence by definition of $L$ their difference is the $0$ matrix, and they are equal.  Therefore a matrix in $L$ is uniquely determined by its last column $y \in \R^q$.  This induces a map $B$ sending $y$ to the first $k$ columns; that is, $y \mapsto B(y) \colon \R^q \to \Hom(\R^k, \R^q)$.   Then the matrix obtained by appending $y$ as the final column of $B(y)$ is precisely $A(y)$.

To see that the fibration is global, observe that a point $t = (t_1,t_2) \in \R^k \times \R^q = \R^{k+q}$ is covered by the fiber through $y \in \R^q$ if and only if $t_2 = B(y)t_1 + y$.  Since $B(y)t_1 + y = A(y)(t_1,1)$, such $y$ exists because the map $A(\cdot)(t_1,1)$ is a linear isomorphism.  Moreover, since this $y$ is unique, the fibration is skew.

Now $A \colon \R^q \to \Hom(\R^{k+1},\R^q)$ is linear, so nonsingularity of $dA \colon \R^q \to \Hom(\R^{k+1},\R^q)$ follows from that of $A$, and so the fibration is nondegenerate.
\end{example}

\section{Continuity at infinity}
\label{sec:contatinf}

One of the most important properties exhibited by skew fibrations is best described as \emph{continuity at infinity}.   Let $\pi \colon \R^n \to U \subset \grass$ be a skew fibration.   Let $L \subset S^{n-1}$ be the set consisting of all oriented directions which appear in fibers of the fibration.  That is, 
\[
L = \bigcup_{u \in U} u \cap S^{n-1}.
\] 
Let $\ell \in L$ be an oriented line contained in fiber $u$.  Informally, Continuity at Infinity asserts that, if $y_n \in \R^n$ is a diverging sequence of points with ``limiting direction'' $\frac{y_n}{|y_n|} = \ell$, then the corresponding $k$-plane fibers $u_n \coloneqq \pi(y_n) \in U$ approach $u$.

Besides its importance as a key feature of skew fibrations, the Continuity at Infinity property is essential for formalizing the correspondence between skew and sphere fibrations.   Consider a fibered $\R^n$ positioned as the tangent hyperplane at the north pole of $S^n$.  The inverse central projection induces a fibration of the open upper hemisphere by open great $k$-hemispheres.  It seems natural to ``complete'' these hemispheres and consider the resulting covering by great $k$-spheres, but the fact that these hemispheres extend continuously to the equator of $S^n$ is surprisingly nontrivial.  This fact is a reformulation of Theorem \ref{thm:continf}, which makes precise the notion of Continuity at Infinity. 

We require some additional terminology.  Recall that, for a skew fibration $\pi \colon \R^n \to U$, there is a natural embedding of the base space into $\R^n$, as the set $v(U)$ consisting of points which are nearest to the origin; that is, the skew fibration can be written as $v \circ \pi : \R^n \to v(U)$.   Let $\zeta$ represent the spherization of this bundle.   Then $\zeta$ is a topologically trivial bundle $v(U) \times S^{k-1}$ embedded in $R^n$: the fiber $S^{k-1}$ lying over a point $v(u)$ is the unit sphere in the $k$-dimensional fiber $(u,v(u))$, centered at $v(u)$.

Now consider the map $g \colon \zeta \to S^{n-1}$, defined by $g(x) = x - v(\pi(x))$.  This is the position vector of $x$ with respect to the basepoint $v(\pi(x))$.  For example, if we restrict $g$ to the unit sphere in a single fiber $(u,v(u))$, the image is the set of oriented directions $S^{n-1} \cap u$ contained in this fiber.  Now $g$ is a continuous injective map from the $(n-1)$-dimensional manifold $\zeta$ to $S^{n-1}$ and is therefore a homeomorphism onto its image $L$.

\begin{lem}
The function $\iota:L \to U$ sending an oriented line to its containing fiber is continuous.
\label{lem:contdir}
\end{lem}

\begin{proof} The function $\iota$ is equal to the composition $g^{-1} \circ \pi\big|_\zeta$.
\end{proof}

\begin{thm}[Continuity at Infinity]
\label{thm:continf}
Let $\pi \colon \R^n \to U$ be a skew fibration.  Let $\ell \in L$ and let $u = \iota(\ell)$. Let $y_n \in \R^n$ be a sequence of points and $(u_n,v_n)$ their containing fibers.  If $|y_n| \to \infty$ and $\frac{y_n}{|y_n|} \to \ell$, then $u_n \to u$.
\end{thm}

\begin{proof}
It is convenient to introduce the following cone for $N \geq 1$ and $0 < \delta < \frac{\pi}{2}$:
\begin{align}
\label{eqn:conedef}
\Cee_{N,\delta} \coloneqq \left\{ y \in \R^n : \langle y, \ell \rangle \geq N, \ \operatorname{dist}_{S^{n-1}} \left( \frac{y}{|y|} , \ell \right) \leq \delta \right\}.
\end{align}
Then by the hypotheses of the theorem, each $\Cee_{N,\delta}$ contains all but a finite number of the $y_n$.  

\begin{claim}
For any closed neighborhood $D \subset U$ of $u$, there exist $N, \delta$ such that $\pi(\Cee_{N,\delta}) \subset D$.
\end{claim}

A proof of the claim establishes that all points inside $\Cee_{N,\delta}$ have fibers from $D$, and since each $\Cee_{N,\delta}$ contains all but finitely many $y_n$, $D$ contains all but finitely many $u_n$, hence $u_n \to u$.  Therefore the claim implies the theorem.

To prove the claim, we need only formalize the following geometric idea: the foliated neighborhood $\pi^{-1}(D)$, which surrounds the fiber $(u,v)$, eventually consumes some $\Cee_{N,\delta}$.  This is depicted in Figure \ref{fig:cone}.

\begin{figure}[h!t]
\centerline{
\includegraphics[width=4in]{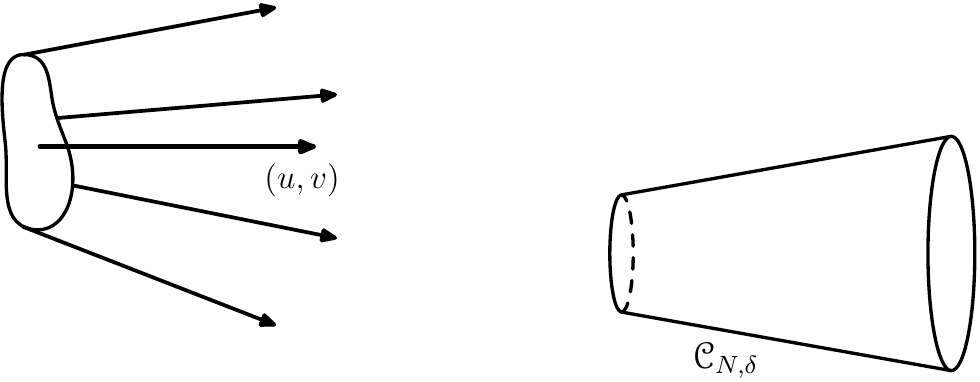}
}
\caption{The set of points with fibers from $D$ eventually consumes some $\Cee_{N,\delta}$}
\label{fig:cone}
\end{figure}

For any closed neighborhood $D \subset U$ of $u$, continuity of $\iota:L \to U$ (Lemma \ref{lem:contdir}) implies that the preimage $\iota^{-1}(D)$ is a closed neighborhood of $\ell \in L$, and therefore contains a closed ball $B \coloneqq B_{2r}(\ell)$ for some $2r <\pi/2$.  Let $u' \coloneqq \iota(\ell')$ for $\ell' \in B$.  We define the map
$$f: B \times \R \longrightarrow \R^n : (\ell',t) \longmapsto v(u') + \frac{- \langle v(u'), \ell \rangle + t}{\langle \ell', \ell \rangle} \ell'.$$
Note that the function $f$ is the linear combination of $v(u')$ with some scalar multiple of $\ell'$, and therefore for fixed $\ell'$, the image of $f$ is contained within the fiber $u'$.  In general, this yields the containment: $\operatorname{Image}(f) \subset \pi^{-1}(D)$.  Thus to show the claim, it suffices to show that $\Cee_{N,\delta} \subset \operatorname{Image}(f)$ for some $N, \delta$.  For this purpose, the function $(\ell',t) \longmapsto v(u') + t\ell'$ would work equally well, but $f$ is more convenient because the quantity
\begin{align}
\label{eqn:fjust}
\langle f(\ell',t) , \ell \rangle = t
\end{align}
is independent of $\ell'$.  Thus for fixed $t$, $f(B,t)$ is contained in the hyperplane $P_t$ which is orthogonal to $\ell$ and distance $t$ from the origin.  We will use this fact momentarily, but first let us make one more observation.  Let $p:\R^n \setminus \left\{ 0 \right\} \to S^{n-1}$ be the normalization $y \mapsto \frac{y}{|y|}$.  For fixed $\ell'$, $f$ parametrizes a line with direction $\ell'$, therefore
$$\lim_{t \to \infty} p \big(f(\ell',t)\big) = \lim_{t \to \infty} \frac{f(\ell',t)}{|f(\ell',t)|} = \ell',$$
and since $B$ is compact, the limit is uniform: there exists $N' \geq 1$ such that 
\begin{align}
\label{eqn:distapprox}
\operatorname{dist}_{S^{n-1}}\left( p \big(f(\ell',t)\big), \ell' \right) < r \hspace{.1in} \mbox{for all } \ell' \in B \mbox{ and } t \geq N'.
\end{align}
We now use equations (\ref{eqn:fjust}) and (\ref{eqn:distapprox}) to show the following claim, which will complete the proof of the theorem.

\begin{claim}
$\Cee_{N',r} \subset \operatorname{Image}(f)$.
\end{claim}

Consider the plane $P_t$ defined above.  To show the claim, it suffices to show that for all $t \geq N'$, there is containment in each cross-section:
\begin{align*}
P_t \cap \Cee_{N',r} \hspace{.05in} \subset \hspace{.05in} P_t \cap \operatorname{Image}(f) = f(B,t),
\end{align*}
where the equality follows from (\ref{eqn:fjust}).

For $t \geq N' \geq 1$, the restriction $p \big|_{P_t}$ is a homeomorphism, so we can instead show the following inclusion for all $t \geq N'$:
\begin{align*}
B_r(\ell) = p \big( P_t \cap \Cee_{N',r} \big) \hspace{.05in} \subset \hspace{.05in} p \big( f(B,t) \big),
\end{align*}
where the equality follows from the definition (\ref{eqn:conedef}).

To show that $B_r(\ell) \subset p \big( f(B,t) \big)$, it is enough to show that $p\big(f(\partial B, t)\big)$ encloses $B_r(\ell)$.  This follows from applying (\ref{eqn:distapprox}) to $\ell' \in \partial B = \partial B_{2r}(\ell)$.  Indeed, for such points $\ell'$, $p \big(f(\ell',t)\big)$ is within distance $r$ of $\ell'$, and hence more than distance $r$ from $\ell$. 
\end{proof}

The idea of the proof also yields the following useful statement, which is an immediate consequence of the first claim in the proof.  Consider a $k$-plane $u$ in $\R^n$ and a skew fibration of some open neighborhood of $u$.  Then for any line $\ell$ contained in $u$, this open foliated neighborhood contains some cone $\Cee_{N,\delta}$.  Therefore, any line in $\R^n$ with direction $\ell$ will intersect the foliated neighborhood.  We have shown the following.

\begin{namedthm}[\ref{thm:lg}]
If a fibration of $\R^n$ by copies of $\R^k$ is locally skew, then it is globally skew.
\end{namedthm}

Local skewness for nondegenerate fibrations was shown in Section \ref{sec:nondegenerate}, and therefore we have the following.

\begin{cor}
\label{cor:globallyskew}
A nondegenerate fibration of $\R^n$ is skew.
\end{cor}

\section{Skew and sphere fibrations}
\label{sec:skewsphere}

Any fibration of $S^n$ by oriented great spheres $S^k$ induces a skew fibration by central projection.   Here we study when this process can be reversed, including whether continuity and smoothness are preserved.

\subsection{From skew to spherical fibrations}
\label{sec:skewspheretop}

Let $\pi \colon \R^n \to U$ be a skew fibration, and position the fibered $\R^n$ as the tangent hyperplane of $S^n \subset \R^{n+1}$, say at the north pole $(0,\dots,0,1)$.  The inverse central projection induces a fibration of the open upper hemisphere $S^n_+$ by open great $k$-hemispheres.  We would like to extend each hemisphere to a great sphere.  
This ultimately results in a covering of the open upper and lower hemispheres $S^n_+$ and $S^n_-$, as well as the subset $L \subset S^{n-1}$ of the equator consisting of those lines which appear in fibers of the skew fibration.   To see this, observe that the projection and subsequent extension of an affine fiber $(u,v)$ results in the great $k$-sphere $S_{(u,v)} = S^n \cap \alpha(u,v)$, where the embedding $\alpha \colon \affgrass \to \grassy$ was introduced in Section \ref{sec:grassmann}.   We define
\[
W = S^n_+ \cup S^n_- \cup L = \bigcup_{u \in U} S_{(u,v(u))}
\]
to be the set of points covered in this manner.

\begin{lem}
\label{lem:Wfibration}
Let $\pi \colon \R^n \to U$ be a skew fibration, let $v \colon U \to \xi_U \subset \affgrass$ be the induced section of the canonical bundle over $\grass$, and let $p \colon W \to \alpha(v(U))$ be the map sending a point of $W$ to its containing $(k+1)$-plane as described above.  Then $p$ is a great sphere fibration of $W$.
\end{lem}

\begin{proof}
On the upper hemisphere $S^n_+$, the map is given by $p\big|_{S^n_+} = \alpha \circ v \circ \pi \circ c$, where $c$ is the central projection.   The restriction $p\big|_{S^n_-}$ is similar, just composed with the antipodal map.  Thus it remains to show continuity on $L \subset S^{n-1}$.

Let $w_n \to \ell \in L$.   For points $w_n \in L$, convergence of the images follows from \ref{lem:contdir}.  For $w_n \in S^n_+$, let $y_n = c(w_n)$.   Convergence of $w_n \to \ell$ has two consequences: $\frac{y_n}{|y_n|} \to \ell$ and $|y_n| \to \infty$.  These are precisely the hypotheses of Theorem \ref{thm:continf}, so convergence of the images follows from Continuity at Infinity.
\end{proof}

A surprising consequence is that great $k$-spheres obtained as ``limits'' of spheres in the fibration $p \colon W \to \alpha(v(U))$ are disjoint from $W$.

\begin{lem}
\label{lem:pnonintersect}
Let $\pi \colon \R^n \to U$ be a skew fibration, and let $(u_n,v_n)$ be a sequence of fibers satisfying $u_n \to \partial U$.  Let $S = \lim S_{(u_n,v_n)}$, passing to a convergent subsequence if necessary.  Then $S \subset S^{n-1}$ is disjoint from $L$.
\end{lem}

\begin{proof} Since $u_n \to \partial U$, $|v_n| \to \infty$ (Lemma \ref{lem:closed}), so $S = \lim S_{(u_n,v_n)}$ is indeed a subset of the equator $S^{n-1}$.  We must show that $S$ does not intersect $L$.

Assume for contradiction that there exists $\ell' \in L \cap S$.  Since $\ell' \in L$, there is a corresponding flat fiber $(u',v') \coloneqq (\iota(\ell'), v(\iota(\ell')))$ and hence a great sphere $S_{(u',v')}$.  On the other hand, since $\ell' \in S = \lim S_{(u_n,v_n)}$, there exist $w_n \in S_{(u_n,v_n)}$ with $w_n \to \ell'$.  By Lemma \ref{lem:Wfibration}, $w_n \to \ell'$ implies that $S_n \to S_{(u',v')}$.  Therefore $S_{(u',v')} = S$.  This is a contradiction, since $S$ lies entirely in the equator while $S_{(u',v')}$ does not.
\end{proof}

\begin{namedthm}[\ref{thm:tf}]
Given a fibration of $\R^n$ by skew oriented affine copies of $\R^k$, there exists a copy of $\R^{k+1}$ transverse to all fibers.
\end{namedthm}

Recall the convention here that two affine subspaces are $\R^n$ are transverse if their linear translates intersect only at the origin (there is no requirement that the planes span $\R^n$).

\begin{proof}
By Lemma \ref{lem:pnonintersect}, there exists a great $k$-sphere in $S^{n-1}$ which does not intersect $L$.  This corresponds to a $(k+1)$-plane which is transverse to all fibers.
\end{proof}

It follows that the base space of a fibration of $\R^{2k+1}$ by skew oriented affine $k$-planes must be homeomorphic to $\R^{k+1}$.  Thus the base spaces of any two such fibrations are homeomorphic, and such fibrations are topologically trivial, so we obtain the following result.

\begin{namedcor}[\ref{cor:quesa}]
Given a continuous (resp.\ smooth) fibration of $\R^{2k+1}$ by skew oriented affine copies of $\R^k$, there exists a fiber-preserving homeomorphism (resp.\ diffeomorphism) of $\R^{2k+1}$ taking the fibration to a Hopf fibration.
\end{namedcor}

\begin{thm}
\label{thm:equivalent}
Let $\pi \colon \R^{2k+1} \to U$ be a fibration by skew oriented affine $k$-planes.  The following are equivalent:
\begin{compactenum}
\item $\pi$ corresponds via central projection to an oriented great $k$-sphere fibration of $S^{2k+1}$.
\item The set $L \subset S^{2k}$ of all oriented lines appearing in fibers of $\pi$ is the complement of a great $k$-sphere.
\item There is a unique linear subspace $\R^{k+1} \subset \R^{2k+1}$ transverse to all fibers.
\item Any map $B \colon \R^{k+1} \to \Hom(\R^k,\R^{k+1})$ corresponding to $\pi$, considered as a map $\R^{k+1} \times \R^k \to \R^{k+1}$, $(y,t) \mapsto B(y)t$, has the property that $B(\cdot)t$ is surjective for every fixed nonzero $t$.
\end{compactenum}
\end{thm}

\begin{proof}

We begin with the simpler implications.

``$(2) \Longleftrightarrow (3)$":
If $L = S^{n-1} \setminus S$, then $S$ consists of the only directions not contained in the fibration, hence it represents a unique transverse $\R^{k+1}$.  Conversely,  let $P = \R^{k+1}$ be transverse to all fibers,  and let $S = S^{n-1} \cap P$.   We will show the contrapositive.  Suppose that there exists $\ell \in S^{n-1} \setminus (S \cup L)$.  Then there exists a sequence of points $w_n \in S^n_+$ converging to $\ell$.  Corresponding to $w_n$ there exist $k$-planes $(u_n,v_n)$ with (some subsequence of) $u_n$ converging to $\partial U$.  Then by Lemma \ref{lem:pnonintersect}, $\lim S_{(u_n,v_n)}$ corresponds to a copy of $\R^{k+1}$ transverse to all fibers.   Since $\ell \in \lim S_{(u_n,v_n)}$ and $\ell \notin S$, this $(k+1)$-plane is distinct from $P$.

``$(2) \Longleftrightarrow (4)$": 
Assume $L = S^{n-1} \setminus S$.  Then there is a unique transverse $\R^{k+1}$, and any map $B$ which globally represents $\pi$ must be defined on this plane.  For fixed $y$ and fixed nonzero $t$, the ordered pair $(t,B(y)t) \in \R^n$ corresponds to a line contained in the fiber through $y$, hence some element of $S^{n-1} \setminus S$.   Since $L =  S^{n-1} \setminus S$, every possible $(t,\cdot)$ appears, hence $B(\cdot)t$ is surjective.  Conversely, if $B(\cdot)t$ is surjective for all nonzero $t$, then all elements of $S^{n-1} \setminus S$ appear as directions in $\pi$, where here $S$ is the $k$-sphere corresponding to the $(k+1)$-plane on which $B$ is defined.

``$(1) \Longrightarrow (2)$":  
Let $\pi \colon \R^{2k+1} \to U$ be a skew fibration of $\R^{2k+1}$ which corresponds via central projection to a great $k$-sphere fibration of $S^{2k+1}$.   It suffices to show that every equator $S^{2k}$ contains exactly one fiber.  There cannot be two fibers, since two great $k$-spheres intersect in $S^{2k}$.  If there is no such fiber, then every great sphere intersects $S^{2k}$ transversely.  For $k=3$ or $k=7$, this yields a great sphere fibration of $S^{2k}$ by great $(k-1)$-spheres, which cannot exist.  For $k=1$, this yields a fibration of $S^3$ by great circles with base space $\R P^2$, impossible since the base space of any such fibration is homeomorphic to $S^2$.

``$(2) \Longrightarrow (1)$": 
Let $\pi \colon \R^{2k+1} \to U$ be a skew fibration of $\R^{2k+1}$, and assume the set of directions $L \subset S^{2k}$ is the complement of some $k$-sphere $S$.  By Lemma \ref{lem:Wfibration}, the inverse central projection gives a fibration of $S^n \setminus S$ by the great $k$-spheres $S_{u,v}$.  We must show that adding $S$ results in a continuous great sphere fibration of $S^n$:

\begin{claim}
If $w_n \in W$ is a sequence of points converging to some point $w \in S$, then the sequence of great spheres $S_n$ through $w_n$ converges to $S$.
\end{claim}

Corresponding to $w_n$ there exist fibers $(u_n,v_n)$ (that is, $S_n = S^{2k} \cap \alpha(u_n,v_n)$) with $u_n$ converging to $\partial U$.   Therefore by Lemma \ref{lem:pnonintersect} there is a corresponding limiting sphere which is disjoint from $L$.  Since $S$ is the only great $k$-sphere disjoint from $L$, it must be the limit of $S_n$.
\end{proof}

\subsection{Nondegenerate and sphere fibrations}

We continue the investigation in the smooth category, beginning with a cautionary reminder. Previous errors in the theory of great sphere fibrations have concerned surprising issues of smoothness.   Yang studied the relationship between division algebras and fibrations and incorrectly claimed smoothness of fibrations constructed in terms of certain division algebras.  %In the current language, the incorrect statement would have asserted that if a smooth nondegenerate fibration is given by a nonsingular bilinear map, then in addition to Continuity at Infinity we have Differentiability at Infinity (and hence the resulting sphere fibration is smooth).   WRONG! These are regular algebras but not division algebras.
Grundh{\" o}fer and H{\" a}hl \cite{GrundhoferHahl} corrected this mistake as follows:

\begin{thm}[Grundh{\" o}fer and H{\" a}hl \cite{GrundhoferHahl}]
\label{thm:gh}
The fibration determined by a real division algebra $D$ of finite dimension is a differentiable locally trivial fiber bundle if and only if $D$ is isomorphic to $\R$, $\C$, $\Ham$, or $\Oct$; that is, if and only if the resulting great sphere fibration is Hopf.
\end{thm}

Gluck, Warner, and Yang \cite{GluckWarnerYang} established that a smooth, closed, connected submanifold $M \subset \grassy$ is the base space of a smooth fibration of $S^n$ by great sphere copies of $S^k$ if and only if for all $P \in M$, $M$ is transverse to the bad cone $B_P$ (see also \cite[Proposition 1]{CahnGluckNuchi}).

In Lemma \ref{lem:Wfibration} we showed that if $W \subset S^{n}$ is the subset of $S^n$ obtained by extending the open hemispheres which result from the inverse central projection of a skew fibration, then $W \to \alpha(v(U))$ is a great sphere fibration.  Thus the following statement follows immediately from Theorem \ref{thm:globalclass} and the classification of great sphere fibrations above:

\begin{prop}
Let $\pi \colon \R^n \to U$ be a skew fibration and $p \colon W \to \alpha(v(U))$ the corresponding great sphere fibration of $W \subset S^n$.  Then $\pi$ is a smooth nondegenerate fibration if and only if $p$ is a smooth fibration.
\end{prop}

Even assuming that $p$ extends to a great sphere fibration of $S^n$, it seems difficult to write a condition which captures smoothness of this extension and is not tautological.  For example, consider a great circle fibration of $S^3$ which is smooth on the complement of one fiber.  Projecting to any tangent space which is parallel to this fiber yields a smooth nondegenerate line fibration of $\R^3$.

\section{Results for skew and sphere fibrations}
\label{sec:applications}

Here we state and prove a number of results for skew fibrations using the machinery developed in the previous sections.  When applicable, we discuss how such results are related to those for great sphere fibrations.

\subsection{Germs of great circle fibrations} We begin with a recent result of Cahn, Gluck, and Nuchi, who showed in \cite{CahnGluckNuchi} that every germ of a smooth fibration of $S^{2m+1}$ by oriented great circles extends to such a fibration of all of $S^{2m+1}$.  The idea is as follows: let $M \subset G_2(2m+2)$ be the base space of a great circle fibration of $S^{2m+1}$. Then for $P \in M$, the tangent space $T_P M$ corresponds to a linear map $\R^{2m} \to \R^{2m}$ with no real eigenvalues.   To show that each such linear map can play the role of the tangent space to $P$ for some smooth fibration,  it is possible to construct a smooth fibration which interpolates between a well-chosen Hopf-like fibration away from $P$ and a (local near $P$) fibration which has a prescribed tangent space at $P$.  This can then be perturbed near $P$ to match a prescribed germ, since the no-real-eigenvalue condition is open.

The interpolation is necessary because it may not be possible to define a global fibration using the prescribed map with no real eigenvalues, and even if such a fibration exists,  Theorem \ref{thm:gh} tells us that the resulting fibration might not be smooth.  However, there is no such issue for nondegenerate fibrations of $\R^n$, allowing for a simple proof of the corresponding statement, even for skew fibrations with higher-dimensional fibers.

\begin{namedthm}[\ref{thm:germ}]
Every germ of a smooth nondegenerate fibration of $\R^n$ by $\R^k$ extends to such a fibration of all of $\R^n$.
\end{namedthm}

\begin{proof}
Suppose we begin with the germ of a smooth nondegenerate fibration of $\R^n$, represented by a map $B \colon E \subset \R^{n-k} \to \Hom(\R^k,\R^{n-k})$, defined on a neighborhood $E$ of $0$ in $\R^{n-k}$.   Then $dA_0$ is a nonsingular bilinear map and hence yields a smooth nondegenerate fibration by the construction in Example \ref{ex:bilinear}.  Specifically, this fibration is given by $B' = dB_0 \colon \R^{n-k} \to \Hom(\R^k,\R^{n-k})$.  Now a small perturbation of $B'$ can be made to match $B$ in some neighborhood of $0$.  Since the nonsingularity condition is open, and since this perturbation does not affect the map near $\infty$ of $E$, Theorem \ref{thm:nonclass} guarantees that this perturbation yields a smooth fibration of $\R^n$.
\end{proof}

\subsection{Deformation of fibrations}  We have seen that a linear map $B \colon \R^{2m} \to \R^{2m}$ corresponds to a smooth nondegenerate line fibration of $\R^{2m+1}$ if and only if $B$ has no real eigenvalues.  According to McKay \cite[Lemma 6]{McKay}, the space of linear maps with no real eigenvalues deformation retracts to its subspace of linear complex structures (that is, linear maps which square to the negative identity), which in turn deformation retracts to its subspace of orthogonal complex structures (see e.g.\ \cite[Proposition 2.5.2]{McDuffSalamon}).   The next proposition follows immediately.

\begin{prop} The subspace of nondegenerate line fibrations of $\R^{2m+1}$ generated by linear maps $B \colon \R^{2m} \to \R^{2m}$ deformation retracts to its subspace of Hopf fibrations.
\end{prop}

Though trivial in light of prior work, this result reduces the problem of deformation retracting the space of nondegenerate line fibrations to its subspace of Hopf fibrations to one of deformation retracting the space of nondegenerate line fibrations to its subspace of nondegenerate line fibrations given by linear maps.

\subsection{Great circle fibrations and contact structures}

In \cite{Gluck},  Gluck showed that the plane distribution orthogonal to any great circle fibration on $S^3$ is contact.  Then Gluck and Warner's classification, combined with Gray stability, can be used to show that any such contact structure is tight.  In \cite{HarrisonBLMS}, we showed that the plane distribution orthogonal to any nondegenerate line fibration of $\R^3$ is a tight contact structure.  The study of contact structures corresponding to (not necessarily skew) line fibrations of $\R^3$ was continued by the author \cite{HarrisonAGT} and Becker and Geiges \cite{BeckerGeiges}.  

In \cite{GluckYang}, Gluck and Yang construct examples of great circle fibrations on $S^{2m+1}$, $m \geq 2$, which do not correspond to contact structures.  The same construction can be used to directly construct nondegenerate line fibrations of $\R^{2m+1}$ which do not correspond to contact structures.

Define $B \colon \R^{2m} \to \Hom(\R,\R^{2m}) = \R^{2m}$ to be the linear map represented in the standard basis by a matrix $M$ with $2 \times 2$ blocks
\[
\left( \begin{array}{cc}
0 & \frac12 \\ -\frac12 & 0
\end{array}
\right)
\]
down the diagonal, and a single $2 \times 2$ identity matrix in the top right corner.  The details leading to this definition can be seen in \cite{GluckYang}, but the important observations are that:
\begin{compactitem}
\item the linear map $B$ has no real eigenvalues, and
\item the matrix $M - M^{tr}$ is not invertible.
\end{compactitem}

Now $B$ is linear, so for any $y \in \R^{2m}$, the linear map $dB_y$ can also be represented by $M$.  Thus for all $y \in \R^{2m}$, $dB_y$ has no real eigenvalues,  and hence $B$ corresponds to a smooth fibration of $\R^{2m+1}$ by oriented lines. 

Let $V$ be the unit vector field which defines the line fibration.   We would like to show that the $1$-form $\alpha$ dual to $V$ is not contact at $0$.   For $y \in \R^{2m}$ and standard coordinates $(y_1,\dots,y_{2m}, y_{2m+1})$, we write
\[
\alpha(y) = \frac{1}{|B(y)|^2 + 1}\left( B_1 dy_1 + \cdots + B_{2m} dy_{2m} + dy_{2m+1} \right).
\]
A short computation shows that $d\alpha_0 \big|_{\R^{2m}}$ is represented by $M - M^{tr}$.   Since this matrix is not invertible, $d\alpha_0 \big|_{\R^{2m}}$ is degenerate.

We have shown:

\begin{namedthm}[\ref{thm:contact}]
For odd $n \geq 5$, there exists a smooth nondegenerate line fibration of $\R^n$ such that the orthogonal hyperplane distribution is not a contact structure.
\end{namedthm}

The argument above can lead to a similar example on $S^{2m+1}$.  The differential of the inverse central projection at the point $0$, which sends $\R^{2m+1}$ to itself, in the sense that $\R^{2m+1}$ serves as the tangent space both to $0$ in $\R^{2m+1}$ and to the north pole in $S^{2m+1}$, is the identity.  Thus the corresponding hyperplane distribution is not contact there.  The induced great circle fibration does not complete to a great circle fibration of all of $S^{2m+1}$, but the extension result of Cahn, Gluck, and Nuchi may be applied to achieve an extension.  In spirit this is the argument of Gluck and Yang, though seen from the perspective of the associated line fibration.

The relationship between nondegenerate fibrations and contact structures can be explored in somewhat greater detail.   If $B \colon \R^{2m} \to \R^{2m}$ corresponds to a nondegenerate fibration, then for all $y \in \R^{2m}$, $dB_y$ has no real eigenvalues, and the contact condition can be easily studied if $dB_y$ is represented by a matrix $M_y$ in Jordan normal form.  For example, if every $dB_y$ is diagonalizable over $\C$, then every $M_y - M_y^t$ is invertible, and $B$ corresponds to a contact structure.  When $dB_y$ is not diagonalizable, the contact condition is a condition on the value of the imaginary part of those eigenvalues for which the geometric and algebraic multiplicity do not match.   For example, the examples above arise by observing that an eigenvalue with imaginary part $\frac12$ and defect $1$ causes noninvertibility of the matrix $M - M^{tr}$.  The corresponding inadmissible imaginary values for higher defects may be computed, but it is unclear whether there is a nice formula for the values in terms of the defect.   Nevertheless, it is clear that there exist precise pointwise conditions on a unit vector field $V$ corresponding to a line fibration which dictate whether the line fibration is contact.  It would be interesting to study tightness questions for such fibrations.

\subsection{Great circle fibrations from nonsingular bilinear maps}
\label{sec:gcnon}  We have seen two examples for which statements about skew fibrations followed more easily than their spherical counterparts, primarily due to the ability to construct skew fibrations globally using nonsingular bilinear maps (that is, linear maps with no real eigenvalues).

We now address the possibility of obtaining great circle fibrations of $S^{2m+1}$ by the inverse central projection of nondegenerate line fibrations of $\R^{2m+1}$.  In Theorem \ref{thm:equivalent} we gave explicit conditions, in the case $n=2k+1$, for when a skew/nondegenerate fibration corresponds via inverse central projection to a great sphere fibration.  The problem when $n \neq 2k+1$ is more difficult, because the inverse central projection must leave more than a single fiber uncovered. 

Suppose that a linear map $B \colon \R^{2m} \to \R^{2m}$ corresponds to a nondegenerate line fibration of $\R^{2m+1}$; in particular, this occurs if and only if $B$ has no real eigenvalues.  Then $B$ is invertible,  so each direction transverse to $E = \R^{2m}$ appears in the fibration.  Therefore the inverse central projection and subsequent completion to great circles covers $S^{2m+1} - S$, where $S = S^{2m-1}$ is the great sphere obtained from intersecting $S^{2m+1}$ with the translate of $E$ through the origin of $\R^{2m+2}$.

Now we see how we must define the fibration on $S$.  Let $u \in S$ and consider the sequence of points $x_n = nu \in E$, $n \in \N$.  The inverse central projection yields a sequence convering to $u$.   Moreover,  the inverse central projection maps the line through $x_n$ to the great circle spanned by the vectors $(nu,0,1)$ and $(B(nu),1,0) = (nB(u),1,0)$ in $\R^{2n+2}$.  The limiting great circle is spanned by $(u,0,0)$ and $(B(u),0,0)$, which are not multiples due to the no-real-eigenvalue condition.  Thus if we hope to achieve a continuous fibration, we must define the great circle through $u \in S$ as that spanned by $(u,0,0)$ and $(B(u),0,0)$.  However, this is a well-defined covering of $S$ if and only if the linear map $B$ leaves every plane of the form $\Span\left\{u,B(u)\right\}$ invariant.

\begin{definition}
We say that a linear map $B \colon \R^{2m} \to \R^{2m}$ with no real eigenvalues is \emph{invariant on planes} if for every nonzero $u \in \R^{2m}$, $B^2(u) \in \Span\left\{u,B(u)\right\}$.
\end{definition}

Examples of such $B$ are given by almost complex structures; that is, linear maps $J$ with $J^2 = -I$. 

We classify the maps $B$ which are invariant on planes.  Let $P_1$ be a plane of the form $\Span\left\{ u , B(u) \right\}$ for some $u \in S^{2n-1}$.  Then the restriction of $B$ to $P_1$ is a linear map $P_1 \to P_1$ and therefore $B\big|_{P_1}$ has two complex eigenvalues.   Now choose $u_2 \in P_1^\perp$, let $P_2 = \Span\left\{ u_2, B(u_2) \right\}$, and repeat this process.  In this way we see that $B$ is diagonalizable with complex eigenvalues $\lambda_1, \dots, \lambda_n$ and their conjugates.  We now show that $\lambda_1 = \cdots = \lambda_n$.

We write $B$ in real Jordan normal form, so it appears with $n$ $2 \times 2$ blocks down the diagonal, each of the form
\[
\left( \begin{array}{cc} a_i & -b_i \\ b_i & a_i \end{array} \right)
\]
for real $a_i$ and $b_i$.
In these coordinates we compute
\begin{align*}
B(B(e_2+e_4)) & = B(-b_1,a_1,-b_2,a_2,0,\dots,0) \\
& = (-2a_1b_1,-b_1^2+a_1^2,-2a_2b_2,-b_2^2+a_2^2,0,\dots,0) \\
& = 2a_1 B(e_2) - (a_1^2+b_1^2) e_2 + 2a_2 B(e_4) - (a_2^2+b_2^2) e_4.
\end{align*}
By the condition that $B$ is invariant on planes, we obtain $a_1 = a_2$ and $b_1 = \pm b_2$.   Similarly we obtain $a_i = a_j$ and $b_i = \pm b_j$ for all $i$ and $j$, and so we obtain $\lambda_1 = \cdots = \lambda_n$.

Thus the linear maps $B$ which generate nondegenerate fibrations and also correspond to great circle fibrations are those with complex eigenvalues $\lambda$ and $\bar{\lambda}$ with multiplicity $n$. 

\begin{prop} A nondegenerate line fibration of $\R^{2m+1}$ given by a linear map $B \colon \R^{2m} \to \R^{2m}$ corresponds to a (continuous) great circle fibration if and only if $B$ is invariant on planes.
\end{prop}

We have shown that such a linear map corresponds to a covering of $S^{2m+1}$ by great circles if and only if $B$ is invariant on planes.  Though seemingly obvious that such a covering is continuous, we illustrate one surprising occurrence when checking continuity.

Recall that Continuity at Infinity for a fiber direction $u \in U$ asserts that if $x_n \in \R^3$ is a sequence of points with limiting direction $u$, then the sequence of fiber directions also approaches $u$.  To show that the extension is continuous, it would be sufficient to show the following ``Continuity at Infinity''-type property for each $u \in \partial U$: if $x_n \in \R^{2m+1}$ is a sequence of points with limiting direction $u$, then the sequence of directions through $x_n$ approaches $B(u)$.  It is easy to see that this holds on the domain of $B$ when $B$ is linear. 

The next example shows that this notion of continuity catastrophically fails, even for the Hopf fibration.

\begin{example}
Consider the linear map $B \colon \R^2 \to \R^2 \colon p \mapsto ip$.  Writing $E$ for the domain of $B$ and considering the decomposition $\R^3 = E \times \R$, $B$ induces the usual Hopf fibration of $\R^3$.   Another presentation of the Hopf fibration is by the vector field $V$ on $\R^3$, $V(x,y,z) = (xz-y,x+yz,1+z^2)$.   We have seen that the inverse central projection to $S^3$, via Continuity at Infinity, yields a fibration of $S^3$ minus a single great circle, which may then be completed to a continuous fibration of $S^3$.

Let $u = (u_1,u_2,0) \in E$ be unit, and consider the sequence $v + tu$, for $t \in \N$ and fixed $v=(v_1,v_2,v_3) \in u^\perp$.  Then $V(v+tu) = ((v_1+tu_1)v_3-(v_2+tu_2),  v_1+tu_1 + (v_2+tu_2)v_3, 1+v_3^2)$, and the limiting direction is $(u_1v_3-u_2,u_1+u_2v_3,0) = v_3u + B(u)$.

Thus by changing the height $v_3$, it is possible to obtain \emph{any} limiting direction $w$ on the circle spanned by $u$ and $B(u)$, as long as the orientation of $\left\{ u,w \right\}$ matches that of $\left\{ u, B(u) \right\}$. 
\end{example}

The example illustrates that the check for continuity of the fibration cannot be done by checking the continuity of the map sending a point on the upper hemisphere of $S^{2m+1}$ to the intersection of its great circle with the equator $S^{2m}$.   However, if we compose this with the projection to the line orthogonal to $u$ in the plane spanned by $u$ and $B(u)$, then we can readily verify continuity.

\begin{proof}
Let $B \colon \R^{2m} \to \R^{2m}$ be a linear map which has no real eigenvalues and is invariant on planes.  The map $B$ induces a fibration of $\R^{2m+1}$.  We consider coordinates $(x_1,\dots,x_{2m+1})$, where the first $2m$ coordinates span the domain $E = \R^{2m}$ of $B$.

Given a point $(z,\tilde{z}) = (z_1,\dots,z_{2m},\tilde{z})$ in the fibered $R^{2m+1}$, we determine the direction of the fiber through $(z,\tilde{z})$ by solving $(z,\tilde{z}) = (x,0) + t(B(x),1)$ over $x \in E$. Thus $x = (\Id + \tilde{z}B)^{-1}z$, which is well-defined since $B$ has no real eigenvalues, and so the fiber through $z$ has direction $(B(\Id + \tilde{z}B)^{-1}z,1)$.

Now we implement the characterization of maps $B$ which are invariant on planes.  Writing $B$ in the real Jordan normal form, each $2 \times 2$ block is equal to
\[
\left(
\begin{array}{cc} a & -b \\ b & a \end{array}
\right)
\]
for some fixed $a,b \in \R$.  Therefore $(\Id + \tilde{z}B)^{-1}$ has $2 \times 2$ blocks
\[
\frac{1}{(1+\tilde{z}a)^2+\tilde{z}^2b^2}
\left(
\begin{array}{cc}
1 + \tilde{z}a & -\tilde{z}b \\ \tilde{z}b & 1+\tilde{z}a
\end{array}
\right),
\]
and finally, $B(\Id + \tilde{z}B)^{-1}$ has $2 \times 2$ blocks
\[
\frac{1}{(1+\tilde{z}a)^2+\tilde{z}^2b^2}
\left(
\begin{array}{cc}
a + \tilde{z}(a^2+b^2) & -b \\ b & a+\tilde{z}(a^2+b^2)
\end{array}
\right).
\]
Thus the fiber through $z$ has direction $((\tilde{z}(a^2+b^2)\Id + B)z,(1+\tilde{z}a)^2+\tilde{z}^2b^2)$.

Now, as usual, we position the fibered $\R^{2m+1}$ as the tangent hyperplane at the north pole of  $S^{2m+1} \subset \R^{2m+2}$.   In coordinates $(x_1,\dots,x_{2m+2})$ on $\R^{2m+2}$, the fibered $\R^{2m+1}$ is the hyperplane $x_{2m+2} = 1$.   The set of directions achieved by this fibration corresponds to an open hemisphere of the equator $S^{2m}$, which we take to be the hemisphere determined by $x_{2m+1} > 0$.  Therefore the inverse central projection leaves $S \coloneqq S^{2m-1} = \left\{ x \ \big| \ x_{2m+1} = x_{2m+2} = 0 \right\}$ uncovered, and we define the great circle through $u \in S$ as that spanned by $u$ and $B(u)$.  

Let $u \in S$, and consider any great circle through $u$ which is transverse to $S$; say $\cos(t)u+\sin(t)v$, where $u = (u_1,\dots,u_{2m},0,0) \in S \subset \R^{2m+2}$, and $(v,v_{2m+1},v_{2m+2})=(v_1,\dots,v_{2m+2})$, with $v_{2m+2} \neq 0$, is orthogonal to $u$.  The central projection maps this great circle to the line $\frac{1}{v_{2m+2}}(v+ \cot(t) u)$ in the fibered $\R^{2m+1}$.  By the above computation, the direction of the fiber through a point on this line, given as a vector in $\R^{2m+2}$, is 
\[
\left(
\begin{array}{c} \frac{v_{2m+1}}{v_{2m+2}^2}(a^2+b^2)(v+\cot(t) u) + \frac{1}{v_{2m+2}}B(v + \cot(t) u) \\
(1+\frac{v_{2m+1}}{v_{2m+2}}a)^2+(\frac{v_{2m+1}}{v_{2m+2}} b)^2 \\
0
\end{array}
\right).
\]
Next we project to the hyperplane orthogonal to $u$, and using $\mathbf{B}$ to represent the projection of $B$ to the hyperplane orthogonal to $u$, we obtain
\[
\left(
\begin{array}{c} \frac{v_{2m+1}}{v_{2m+2}^2}(a^2+b^2)v + \frac{1}{v_{2m+2}}\mathbf{B}(v + \cot(t) u) \\
(1+\frac{v_{2m+1}}{v_{2m+2}}a)^2+(\frac{v_{2m+1}}{v_{2m+2}} b)^2 \\
0
\end{array}
\right).
\]
It is now easy to see that, as $t \to 0$, the limiting direction is $(\frac{1}{v_{2m+2}}\mathbf{B}(u), 0, 0)$, whose normalization agrees with that of $\mathbf{B}(u)$, and continuity is established.

More formally, and with the additional goal of establishing smoothness, we let $S(t)$ be the length of this vector.  We want to compute
\[
\lim_{t \to 0} \frac{1}{t} \left( \frac{1}{S(t)} \left(
\begin{array}{c} \frac{v_{2m+1}}{v_{2m+2}^2}(a^2+b^2)v + \frac{1}{v_{2m+2}}\mathbf{B}(v + \cot(t) u) \\
(1+\frac{v_{2m+1}}{v_{2m+2}}a)^2+(\frac{v_{2m+1}}{v_{2m+2}} b)^2 \\
0
\end{array}
\right) - \frac{1}{|\mathbf{B}(u)|} \left( \begin{array}{c} \mathbf{B}(u) \\ 0 \\ 0 \end{array} \right) \right).
\]
Now
\[
S(t) = \sqrt{C_1 + C_2\cot t + \frac{|\mathbf{B}(u)|}{v_{2m+2}}\cot^2 t},
\]
for some constants $C_1$ and $C_2$, and so
\[
\lim_{t \to 0} tS(t) = \frac{|\mathbf{B}(u)|}{v_{2m+2}}.
\]
Thus to show differentiability, it suffices to show that the coefficient of $\mathbf{B}(u)$, 
\[
\lim_{t \to 0}\left( \frac{\cot t}{t S(t) v_{2m+2}} - \frac{1}{t|\mathbf{B}(u)|} \right) = \frac{1}{v_{2m+2}} \lim_{t \to 0} \frac{1}{t}\left(\frac{\cot t}{S(t)} - \frac{v_{2m+2}}{|\mathbf{B}(u)|} \right),
\]
is finite.   This follows from a straightforward application of L'H{\^o}pital.
\end{proof}

One of the important consequences of this construction of great circle fibrations is a class of great circle fibrations of $S^{2m+1}$ with identical restriction to $S^{2m-1}$.  Indeed, if $B \colon \R^{2m} \to \R^{2m}$ is a linear map which has no real eigenvalues and is invariant on planes, then the great circle fibration on $S$ is defined by: the great circle through $u$ is that spanned by $u$ and $B(u)$.  Now if $B$ consists of $2 \times 2$ blocks
\[
\left( \begin{array}{cc} a & -b \\ b & a \end{array} \right),
\]
then $B = a\Id + bJ$, and therefore the plane spanned by $u$ and $B(u)$ is the same as that spanned by $u$ and $Ju$.

However, the great circles through points on the open upper hemisphere do genuinely depend on the values of $a$ and $b$.  Thus we have explicitly constructed a class of smooth great circle fibrations of $S^{2m+1}$ which restrict to the Hopf fibration on $S^{2m-1}$. 

The motivation for such examples is to reduce the problem of deforming a great circle fibration to the Hopf fibration to the problem of deforming a map $B$ corresponding to a nondegenerate line fibration to a linear map $B$ with no real eigenvalues and invariant on planes, since the latter problem can be carried out on $\R^{2m}$.  This reduction succeeds by eliminating the concern of the behavior of the fibration at infinity.

\section{Further thoughts}
\label{sec:conclusion}

We offer some final thoughts which could guide future research on line and great circle fibrations.

\subsection{Great circle fibrations on $S^5$}

Given a skew line fibration of $\R^5$, the inverse central projection covers the open upper and lower hemispheres, as well as the equatorial set $L = U \cup -U$, where $U \subset S^4 = G_1(5)$ is the base space of the fibration.  Letting $W$ represent this covered portion, the complement $S^5 \setminus W$ is some $3$-manifold in the equator $S^4$.  When can the induced covering be completed to a great circle fibration of $S^5$? Two necessary conditions are the following:
\begin{compactenum}
\item each $u \in \partial U$ corresponds to a unique limiting great circle $S(u)$
\item each pair of limiting great circles is either equal or disjoint.
\end{compactenum}
The second item can fail as follows: consider a skew line fibration of $\R^3$ for which $U \subset S^2$ is not an open hemisphere (a specific example is given in \cite{HarrisonBLMS}).  Then the uncovered portion is some $2$-dimensional subset of the equator, and any two limiting great circles intersect.

We do not know if the first item could fail: in particular, could there exist two sequences $u_n, u_n' \to u \in \partial U$ such that the limiting great circles (guaranteed by Lemma \ref{lem:pnonintersect}) are different?  It seems reasonable to believe that this is possible, since the condition corresponds to some rigidity in the limiting fibers corresponding to divergent sequence of points in the fibered $\R^5$.

In any case, if both items are satisfied, it seems that a great circle fibration of $S^5$ must result.  Every point $w$ uncovered by the original projected fibers is a limit of some $w_n \in S^5_+$, hence corresponds to elements $u_n \in U$ which approach $u \in \partial U$, and since $S(u)$ is well-defined as the limit of the great circles through $u_n$,  it must contain $w$.  Hence the corresponding great circles cover $S^5$, and continuity follows since the circles are added as limiting circles.

%\subsection{On a potential h-principle for nondegenerate fibrations}

%Another possible method of exhibiting a homotopy equivalence, also considered by McKay \cite{McKay} in the great circle case, would be to prove an h-principle for maps $\R^{2m} \to \R^{2m}$ for which the differential has no real eigenvalues.  Define the singularity set $\Sigma$ in the $1$-jet bundle $J^1(\R^{2m},\R^{2m}) = \R^{2m} \times \R^{2m} \times \Hom(\R^{2m},\R^{2m})$ as the triples $(x,y,L)$ for which $L$ has a real eigenvalue.  The complement of $\Sigma$ is open, but we were not successful in our attempts at applying the methods of holonomic approximation nor convex integration to this differential relation.

%For those familiar with the technique of convex integration, we offer a brief demonstration of how Gromov's condition of ampleness can fail.  The eigenvalues of
%\[
%L = \left(
%\begin{array}{cc} a & -1 \\ b & 0
%\end{array}
%\right)
%\]
%are those $\lambda$ satisfying $\lambda^2 - a\lambda + b = 0$, which are nonreal whenever $a^2-4b < 0$.  In the $ab$-plane, this determines the inside of a parabola, for which the convex hull is not the whole plane.

\subsection{Local and global fibrations}

As first observed by Ovsienko and Tabachnikov \cite{OvsienkoTabachnikov}, the dimensions $k$ and $n$ for which skew fibrations exist are independent of whether the fibrations are local or global.   This is not the case for great sphere fibrations.  The following three scenarios occur (for different values of $k$ and $n$):
\begin{compactitem}
\item there exists a fibration of $S^n$ by great $k$-spheres.
\item there exists a fibration of an open set of $S^n$ by great $k$-spheres, but no fibration of $S^n$ by great $k$-spheres.
\item no open set of $S^n$ can be fibered by great $k$-spheres.
\end{compactitem}

For example, if there exists a fibration of $\R^n$ by skew affine copies of $\R^k$,  with $k \notin \left\{ 0, 1, 3, 7\right\}$,  then the inverse central projection induces a fibration of a full measure open set of $S^n$ by great $k$-spheres which does not extend to a fibration of $S^n$.   In particular, if the skew fibration is given by a bilinear map (see Example \ref{ex:bilinear}), then the resulting great sphere fibration covers $S^n \setminus S^{n-k-1}$.

We do not know if there is an interesting study of such fibrations.  In \cite{Petro}, Petro studied great sphere fibrations of products $S^p \times S^q$ embedded in $S^{p+q+1}$, and he positively answered all three parts of Question \ref{ques:main} for great $3$-sphere fibrations of $S^3 \times S^3 \subset S^7$.

For some values of $n$, such as powers of $2$ and $n=80$, there is no skew fibration of $\R^n$, therefore there is no fibration of any open set in $S^n$ by great $k$-spheres for any $k$.   This observation has potential ramifications, since this may act as an obstruction when considering the structure of great sphere fibrations.  For example, consider a great circle fibration of $S^5$.  For any equator $S^4$,  the fact that there is no (local) fibration of $\R^4$ by skew lines implies that no $4$-dimensional subset of $S^5$ may be fibered by great circles, which in turn implies that the union of the great circles lying completely in $S^4$ is not $4$-dimensional (this was known to Gluck, Warner, and Yang \cite[Section 10]{GluckWarnerYang} by a different argument).  Since the great circles lying in the equatorial $S^4$ are precisely those which have no image under the central projection, the structure of this set is of great import.

\bibliographystyle{plain}
\bibliography{../bib}{}

\end{document}